\documentclass[a4paper]{amsart}

\usepackage[english]{babel}
\usepackage[T1]{fontenc}
\usepackage{lmodern}
\usepackage[centering]{geometry}
\usepackage{MnSymbol}

\theoremstyle{plain}

\newtheorem{theo}{Theorem}
\newtheorem{lemm}[theo]{Lemma}

\newtheorem{coro}[theo]{Corollary}

\theoremstyle{definition}

\theoremstyle{remark}
\newtheorem{rema}[theo]{Remark}

\newtheorem{theoA}{Theorem}

\newtheorem{lemmA}[theoA]{Lemma}

\newtheorem{coroA}[theoA]{Corollary}

\usepackage{microtype}
\usepackage{enumerate}
\usepackage{graphicx}
\usepackage{color}
\usepackage{bm}
\usepackage{mathtools}

\usepackage[dvipsnames]{xcolor}
\definecolor{FlatRed}{RGB}{231,76,60}
\definecolor{FlatGreen}{RGB}{46,204,113}
\definecolor{FlatBlue}{RGB}{52,152,219}
\definecolor{FlatYellow}{RGB}{241,196,15}
\colorlet{FlatViolet}{FlatRed!50!FlatBlue}
\colorlet{FlatBrown}{FlatRed!50!FlatGreen}
\colorlet{FlatOrange}{FlatRed!50!FlatYellow}
\colorlet{FlatCyan}{FlatGreen!50!FlatBlue}

\usepackage{enumitem}
\usepackage{textcomp}
\usepackage[normalem]{ulem}

\usepackage{hyperref}
\hypersetup{
    colorlinks,
    linkcolor={FlatRed},
    citecolor={FlatGreen},
    urlcolor={FlatBlue}
}
\usepackage{stmaryrd}

\title{Reinforced loop soup via Wilson's algorithm}

\author{Yinshan Chang}
\address{College of Mathematics, Sichuan University, China}
\email{ychang@scu.edu.cn}

\author{Yichao Huang$^{\ast}$}
\address{School of Mathematics and Statistics, Beijing Institute of Technology, China}
\email{yichao.huang@bit.edu.cn}

\author{Dang-Zheng Liu}
\address{School of Mathematical Sciences, University of Science and Technology of China, China}
\email{dzliu@ustc.edu.cn}

\author{Xiaolin Zeng}
\address{Institut de Recherche Mathématique Avancée, Université de Strasbourg, France}
\email{zeng@math.unistra.fr}

\begin{document}

\begin{abstract}
The goal of this note is twofold: first, we explain the relation between the isomorphism theorems in the context of vertex reinforced jump process, discovered in~\cite{MR4021254,MR4255180} and the standard Markovian isomorphism theorems for Markovian jump processes; second, we introduce the vertex reinforced counterpart of the standard Poissonian loop soup developed by Le Jan~\cite{MR2675000}. To this end, we propose an algorithm that can be viewed as a variant of Wilson's algorithm with reinforcement. We establish the isomorphism theorems for the erased loops and the random walk from this algorithm, and in particular provide a concrete construction of the reinforced loop soup via a random process with a reinforcement mechanism.
\end{abstract}

\maketitle

\section{Introduction}
Kurt Symanzik formulated a framework for Euclidean quantum field theory in~\cite{symanzik1968euclidean}, establishing a profound connection between Euclidean quantum field theory and the structure of classical statistical mechanics. Especially, he introduced the random path expansion for the Green function~\cite{MR782955,MR471757}, which was further developed in~\cite{MR648362,MR719815,MR693227,MR902412,MR678000}. The goal of this current paper is to consider Symanzik's loop expansion idea for Euclidean quantum field theory in the context of the so-called $\text{H}^{2|2}$ supersymmetric hyperbolic sigma model in relation to a class of random processes with reinforcement.

One of the realizations of Symanzik's loop (or ``gas/soup'') expansion for the free field theory is known in the mathematics community under the name of Poissonian loop soup, both in the discrete (random walk loop soup~\cite{MR2255196}, Markovian loop soup~\cite{MR2675000}) and in the continuum (Brownian loop soup~\cite{MR2045953}). These objects are connected to the discrete Gaussian free field and the continuous Gaussian free field~\cite{MR2815763,MR2932978}. Loop soup provides an important toolbox for a non-perturbative approach to studying phenomena in quantum field theory. For instance, the Brownian loop soup in dimension two is used as an essential tool to construct the conformal loop ensembles~\cite{MR2979861} and is thus connected to various models of statistical mechanics (especially the interfaces in two-dimensional critical systems). In the case of the Gaussian free fields, the loops in the loop soup are thrown independently as the whole random process can be realized as a Poisson point process on the space of loops. There is a vast literature in the probability community on the Poissonian or Brownian loop soup, in particular on the study of local times~\cite{MR156383,MR154337}, occupational time fields and isomorphism theorems~\cite{MR585179,MR2120243,MR1813843}, random interlacements~\cite{MR2892408,MR3502602} and the list is far beyond our ability to survey here.

It is interesting and important to extend the theory of Symanzik to interacting Euclidean quantum field theory beyond the Gaussian case. It was also Symanzik's original motivation to consider interacting random paths or random loop ensembles. If there were non-quadratic interaction, e.g. $\varphi^{4}$ term or interaction through fermionic variables, the corresponding loop expansion of Symanzik would no longer be Poissonian, and the standard toolbox of Markov processes has to be extended to implement Symanzik's ideas. Following the loop expansion idea of Symanzik, we introduce and study a non-Poissonian loop soup model with reinforcement, for which the underlying field is a supersymmetric hyperbolic sigma model, first introduced by Zirnbauer~\cite{MR1134935} inspired by the work of Efetov~\cite{efetov1999supersymmetry}. The supersymmetric model of Zirnbauer is also abbreviated as the $\text{H}^{2|2}$-model, because the field takes its value in a hyperbolic space which can be parametrized by two bosonic variables $x,y$ and two anti-commuting fermionic variables $\xi,\eta$. Roughly speaking, the $\text{H}^{2|2}$-model is a supersymmetric quantum field theory with the following nearest-neighbor interaction:
\begin{equation*}
    \sqrt{1+x_i^2+y_i^2+2\xi_{i}\eta_{i}}\cdot\sqrt{1+x_j^2+y_j^2+2 \xi_{j}\eta_{j}}.
\end{equation*}
The $\text{H}^{2|2}$-model has been shown to have a phase transition in its coupling parameter in~\cite{MR2736958,MR2728731} on $\mathbb{Z}^{3}$. Later, it has drawn attention to the probability community because of the unexpected discovery in~\cite{MR3420510} of its relation to the so-called Vertex Reinforced Jump Process~\cite{MR1900324,MR1030727}, a non-Markovian jump process which interacts with its own history. The Vertex Reinforced Jump Process, after a suitable random time change, gains an extra property known as partial exchangeability~\cite{MR556418}, and is therefore a mixture of Markov jump processes in some random environment, the law of which is exactly given by some horospherical coordinate of the $\text{H}^{2|2}$-model. Probabilistic models related to the $\text{H}^{2|2}$-model are disordered systems (statistical mechanics models with random coupling parameters). See Section~\ref{subse:supersymmetric_hyperbolic_sigma_model} below for a brief overview of this model, and~\cite{MR2728731,MR2736958,MR3420510,MR3189433,MR3729620,MR3936158,MR3904155} for a far-from-complete list of references.

In~\cite{MR4021254,MR4255180}, several isomorphism theorems have been discovered as the Vertex Reinforced Jump Process counterparts of the standard Markovian forms of respectively BFS-Dynkin isomorphism, generalized second Ray-Knight theorem, and Eisenbaum's isomorphism. In these three isomorphism theorems, the counterpart of the Gaussian free field in the above-mentioned theorems becomes the $\text{H}^{2|2}$ supersymmetric hyperbolic sigma model. This is an important step towards implementing Symanzik's idea in a supersymmetric quantum field theory with non-trivial interaction, thus with interacting random paths. We feel that it is useful to also investigate Symanzik's loop expansion theory of the $\text{H}^{2|2}$-model as a non-perturbative approach to understanding certain integrable aspects of this supersymmetric hyperbolic sigma model. We propose a candidate for the counterpart of the standard Poissonian loop soup in the case of $\text{H}^{2|2}$-model, which we call the reinforced loop soup. While the reinforced loop soup can no longer be realized as a Poisson point process, one of our main contributions is the mathematical construction of this object, using a variant of Wilson's algorithm~\cite{MR1427525,MR1611693} with a natural reinforcement mechanism. It turns out that many features of the Poissonian loop soup still hold when the Gaussian free field is replaced by the $\text{H}^{2|2}$ supersymmetric hyperbolic sigma model. In particular, we establish a Dynkin-type isomorphism theorem, relating the occupation time field of reinforced loop soup to natural observables of the $\text{H}^{2|2}$ field. To show all these isomorphism theorems, we introduce a key tool called the supersymmetric Bayes formula. This formula allows one to translate quite systematically any standard Markovian isomorphism theorem into its Vertex Reinforced Jump Process counterpart. We now give a summary of our approach and main results.

\subsection{Main results}\label{sub:main_results}
Our contribution in this paper is resumed as follows:
\begin{enumerate}
    \item A reinforced Wilson's algorithm. We conceive a variant of Wilson's algorithm with a reinforcement mechanism, which we use to concretely construct the reinforced loop soup. In short, we replace the Markovian exploration in the standard Wilson's algorithm with the Vertex Reinforced Jump Process, and in such a way obtain a reinforced spanning tree and a reinforced random loop collection. The reinforced random loop collection from this algorithm gives a trajectory explanation to the Dynkin-type reinforced loop soup isomorphism below.
    \item Connection of the Dynkin-type isomorphism theorems for the $\text{H}^{2|2}$-model to the standard Markovian isomorphism theorems. We show that the three isomorphism theorems for the $\text{H}^{2|2}$-model discovered in~\cite{MR4021254,MR4255180} can be obtained as annealed versions of the corresponding standard Markovian Dynkin-type isomorphism theorems, by integrating them in the random environment generated by the Vertex Reinforced Jump Process. This idea is based on the connection between the $\text{H}^{2|2}$-model and the Vertex Reinforced Jump Process discovered by Sabot and Tarrès~\cite{MR3420510}, and our main tool is a supersymmetric Bayes formula.
    \item Dynkin-type isomorphism theorem for the reinforced loop soup. Using the reinforced Wilson's algorithm and our method of proof for the $\text{H}^{2|2}$ isomorphism theorems for the Vertex Reinforced Jump Process, we propose and prove a $\text{H}^{2|2}$ isomorphism theorem for the reinforced loop soup. Indeed, we show that the occupation time field of the loops obtained from the reinforced Wilson's algorithm is the correct $\text{H}^{2|2}$ counterpart in Le Jan's isomorphism theorem for the Poissonian loop soup. This gives a natural construction for a non-Poissonian loop soup model and establishes an integrability result thereof.
    \item Complete reconstruction of the reinforced loop soup with the reinforced Wilson's algorithm. We use the reinforced Wilson's algorithm to give a complete reconstruction of reinforced loop soups with arbitrary parameter $\alpha>0$, based on the idea of a Poisson-Dirichlet decomposition of random loops due to Le Jan. We also extend the $\text{H}^{2|2}$ isomorphism theorem for the reinforced loop soup to higher-dimensional $\text{H}^{2k|2k}$ versions with positive integer $k$.
\end{enumerate}

We hope that the definition and the study of the reinforced loop soup can lead to a better understanding of the $\text{H}^{2|2}$-model, and many interesting questions remain to be addressed. To list a few, our method should give precise information on the reinforced random interlacements in the $\text{H}^{2|2}$-model. Another more challenging direction would be to investigate the continuum limit of the reinforced loop soup, towards a definition of the reinforced Brownian loop soup or reinforced loop ensembles for the $\text{H}^{2|2}$-model.

\subsection{Structure of the paper}
This paper is organized as follows.

In Section~\ref{sec:backgrounds} we gather the necessary backgrounds for this paper, including the standard forms of Wilson's algorithm, the definition of the standard Poissonian loop soup, the supersymmetric fields and the Vertex Reinforced Jump Process. In Section~\ref{sec:wilson_s_algorithm_with_reinforcement} we introduce a variant of Wilson's algorithm with reinforcement mechanism. In Section~\ref{sec:a_supersymmetric_bayes_formula} we present a useful lemma (called supersymmetric Bayes formula) relating supersymmetric free field expectations with the $\text{H}^{2|2}$ expectations. In Section~\ref{sec:isomorphism_theorems_for_the_vertex_reinforced_jump_process}, we give alternative proofs of the three BFS-Dynkin type isomorphisms for the Vertex Reinforced Jump Process. In Section~\ref{sec:LoopSoupIsomorphisms} we use the reinforced Wilson's algorithm to construct the reinforced loop soup occuptional time field and establish the corresponding $\text{H}^{2|2}$ isomorphism theorem thereof. In Section~\ref{sec:some_extensions_of_our_results} we reconstruct the reinforced loop soup process with arbitrary parameter $\alpha > 0$ using the reinforced Wilson's algorithm, and extend the isomorphism theorem for the reinforced loop soup to higher-dimensional supersymmetric hyperbolic sigma models.

\subsection*{Acknowledgement}
Y.C. is supported by National Natural Science Foundation of China (grant number 11701395). Y.H. is partially supported by National Key R\&D Program of China (No. 2023YFA1010103) and by NSFC-12301164, and thanks the Institut de Recherche Mathématique Avancée (IRMA) of the Université de Strasbourg for their kind hospitality. D-Z.L. is supported by the National Natural Science Foundation of China \#12371157 and \#12090012. X.Z. acknowledges the Agence Nationale de la Recherche for their financial support via ANR grant RAW ANR-20-CE40-0012-01 and IRMIA++.

\section{Backgrounds}\label{sec:backgrounds}
We collect some elementary backgrounds for the main objects studied in this paper.

\subsection{Graph Laplacian}
Consider a finite graph of vertex set $V=\{i,j,k,\dots\}$ and non-oriented edges $E(V)=\{(ij)\}$. To each edge $(ij)\in E$ is associated a non-negative \emph{weight} $W_{ij}=W_{ji}\geq 0$. When there is no edge between two vertices $i,j\in V$, it is equivalent in this paper to add the edge $(ij)$ to the edge set $E(V)$ and declare that $W_{ij}=W_{ji}=0$ (and implicitly $W_{ii}=0$). To each configuration of weights on the graph is associated a square symmetric \emph{Laplacian matrix} $\Delta_{W}$ of size $|V|$ defined by
\begin{equation*}
    (\Delta_{W})_{ij}=\begin{cases} -W_{ij} & i\neq j \\ \sum_{k\in V\setminus\{i\}}W_{ik} & i=j \end{cases}.
\end{equation*}

We sometimes consider a distinguished vertex $\delta$ called the \emph{root vertex} of the graph. This can either be done by assigning a vertex to be the root, or equivalently, we can \emph{augment} the graph $(V,E(V),W)$ in the following way: the vertex set is augmented to $\widetilde{V}=V\cup\{\delta\}$, the edge set is augmented to $E(\widetilde{V})=E(V)\cup\{(i\delta), i\in V\}$ and we impose that the augmented edge weights satisfy the condition that at least one of the $\{W_{i\delta}\}_{i\in V}$ is positive (so that $\delta$ is not disconnected from the original graph). The augmented graph will be denoted by $\widetilde{V}, E(\widetilde{V})$, and $\widetilde{W}$ for respectively its vertex set, edge set, and edge weight set. All corresponding definitions on the augmented graph will be decorated with an extra $\widetilde{~~}$.

\subsection{Uniform spanning tree}\label{sec:uniform_spanning_tree}
Consider a finite weighted connected graph $V$. Here, the \emph{connectness} of $V$ means that one can move from any vertex $i\in V$ to another vertex $j\in V$ using only positively weighted edges. A \emph{spanning tree} $T$ is a tree of $V$ (that is, a connected subgraph of $V$ without loops) whose set of vertices is exactly $V$. The \emph{weight} $W(T)$ of the spanning tree $T$ is the product of all the weights of its edges, i.e.
\begin{equation*}
    W(T)=\prod_{(ij)\in E(T)}W_{ij}.
\end{equation*}

When $V$ is finite, the number of spanning trees on $V$ is large (more precisely $|V|^{|V|-2}$ by Cayley's formula), but it remains finite. A random spanning tree $\mathcal{T}$ is called \emph{uniform} if a sample $T$ is chosen with probability proportional to $W(T)$, the weight of $T$. In the literature, $\mathcal{T}$ constructed above is often called \emph{weighted uniform spanning tree} (when no weight is specified, the standard uniform spanning tree is the special case of constant weight $W\equiv 1$): our graphs will always be weighted and we drop the term ``weighted'' in the sequel. Notice that it makes no difference whether $V$ is rooted or not in the above definition, and in sections dealing with Wilson's algorithms, we write a rooted graph as $V$ instead of $\widetilde{V}$ for simplicity and to be conform with existing conventions.

\subsection{The standard Markovian Wilson's algorithm}\label{sec:ClassicalWilson}
Given a finite connected graph $V$ with root vertex $\delta$, the celebrated \emph{Wilson's algorithm}~\cite{MR1427525,MR1611693} generates a random (uniform) spanning tree $T$ of $V$ using a procedure known as the \emph{loop erasure}. For later purposes, we recall the standard Markovian\footnote{The terminology ``standard Markovian'' is used in opposition to ``reinforced'' or ``supersymmetric''.} Wilson's algorithm in two different forms. Our presentation follows~\cite[Chapter~4]{MR3616205}.

\subsubsection{Wilson's algorithm and loop erasure}
The most well-known form of the standard Markovian Wilson's algorithm uses the idea of \emph{loop erasure} (or \emph{cycle erasure}). Let $P:i_0\to i_1\to i_2\to\dots$ be a path in $V$, with $(i_k)_{k=0,1,\dots}$ vertices of $V$ appearing in the order in which they are visited. The loop erased path $\mathrm{LE}(P)$ is the path obtained by \emph{chronologically} erasing loops that appear in $P$. It is clear from the construction that $\mathrm{LE}(P)$ is \emph{self-avoiding}, that is, it does not contain any loop.

The standard Markovian Wilson's algorithm with loop erasure procedure can be described as follows. First, order the vertices arbitrarily, say $V=(i_0,i_1,\dots,i_{|V|-1})$, and declare $\delta=i_0$ to be the root of $V$. Let $T(0)=\{\delta\}$ be the initial tree, and inductively grow the tree $T(i)$ in the following way. If $T(i)$ spans $V$, then we stop. Otherwise, let $x$ be the vertex with the smallest subscript not in $T(i)$, and consider an independent Markov random walk starting from $x$ with edge weights $(W_{e})_{e\in E(V)}$ until the first time it hits $T(i)$. Perform the loop erasure operation on the path of this Markov random walk, and add the self-avoiding branch obtained by loop erasure to the tree $T(i)$. This constructs the next $T(i+1)$, which is a tree by self-avoidedness. The final tree $\mathcal{T}$ is spanning and random, and the observation of Wilson is the following (see~\cite[Theorem~4.1]{MR3616205}):
\begin{lemmA}[Wilson's algorithm with loop erasure]\label{lemm:classicalWilson_loopErasure}
The random spanning tree $\mathcal{T}$ constructed via the loop erasure algorithm is the uniform spanning tree of $V$.
\end{lemmA}
In particular, $\mathcal{T}$ is independent of the choice of the root and of the way we numbered the vertices. The next construction shows that at each step when $T(i)$ is constructed, we can reorder the remaining vertices, and the resulting $\mathcal{T}$ will still be the uniform spanning tree. This last observation (see the paragraph under~\cite[Theorem~4.1]{MR3616205}) seems to be less well known and is sometimes omitted in some textbooks, but will be crucial in our construction of the reinforced version of Wilson's algorithm in Section~\ref{sec:wilson_s_algorithm_with_reinforcement}.

\subsubsection{Wilson's algorithm and cycle popping}\label{subse:wilson_s_algorithm_and_cycle_popping}
We recall an alternative description of the standard Markovian Wilson's algorithm. The key idea is usually called \emph{cycle popping}.

Consider a finite connected graph $V$ with root vertex $\delta$. Suppose that a stack of cards $\{S^{i}_k\}_{k\geq 1}$ is situated at each vertex $i\in V$ except at the root vertex $\delta$, with the card $S^{i}_1$ on top of $S^{i}_2$, then $S^{i}_2$ on top of $S^{i}_3$ etc. On each card, an arrow is given, which tells a walker where to go for the next step. The card is ``popped'' (or removed) and the next card revealed when the walker uses this card to go to the next vertex. If the arrow on the card $S^{i}_k$ is sampled independently according to the transition probability at the vertex $i$ to its neighbors using the weights $(W_{ij})_{j\in V}$, i.e.
\begin{equation*}
    \mathbb{P}[i\to j]=\frac{W_{ij}}{\sum_{k\in V}W_{ik}},
\end{equation*}
then the law of the walker is that of a standard Markov chain with edge weights $(W_{ij})_{i,j\in V}$ and root vertex $\delta$. If further the walker stays at the vertex $i$ with the correct exponential transition time before jumping (that is, an independent exponential variable with parameter $\sum_{k\in V}W_{ik}$), then we obtain the standard continuous Markov jump process on $V$ with edge weights $(W_{ij})_{i,j\in V}$ and killed at the root vertex $\delta$.

Now we introduce the idea of \emph{cycle popping}. Consider a configuration of stacks as above, where we can observe one \emph{visible} card at the top of the stack at each vertex except at the root vertex $\delta$. Connecting each vertex to its neighbor indicated on the visible card, we obtain an \emph{oriented} graph of $V$ rooted at $\delta$. If this oriented graph has no cycle, then it is a (oriented) spanning tree with root $\delta$. If this oriented graph has a cycle, then we ``pop'' this cycle by removing all the visible cards in this cycle, revealing the next cards in the respective stacks. We obtain a new oriented graph, and we continue this cycle popping procedure until it stops. If this algorithm stops, we unorient the oriented spanning tree, and obtain a random spanning tree of $V$ rooted at $\delta$.\footnote{If we want to run Wilson's algorithm on an oriented graph, choose the outgoing weights in the definition of the jump rates above, and do not unorient the final spanning tree.} The standard Markovian Wilson's algorithm with cycle popping is the following statement (see~\cite[Lemma~4.2]{MR3616205}):
\begin{lemmA}[Wilson's algorithm with cycle popping]\label{lemm:classicalWilson_cyclePopping}
The random spanning tree $\mathcal{T}$ constructed via the cycle popping algorithm is the uniform spanning tree of $V$.
\end{lemmA}

In particular, $\mathcal{T}$ is independent of the order in which the cycles are popped. It is important to note that the core argument~\cite[Lemma~4.2]{MR3616205} is deterministic: given any sample of stacks, the order in which the cycles are popped is irrelevant for the final spanning tree. Therefore, the previous loop erasure form of the standard Markovian Wilson's algorithm can be regarded as a convenient choice of popping cycles in a deterministic order, but any algorithm that explores the graph properly and pops every possible cycle is equivalent to the standard Markovian Wilson's algorithm. In Section~\ref{subse:WilsonWithOneExploration} we will introduce an appropriate form of Wilson's algorithm with a reinforcement mechanism, suitable to the purpose of this article.

\subsection{Poissonian loop soup}\label{subse:poissonian_loop_soup}
We now recall the standard Markovian construction of Poissonian loop soup, following the language of~\cite{MR4789605} and~\cite{MR2932978}.

Consider a finite graph $V$ with edge weights $(W_{ij})_{i,j\in V}$. A \emph{based loop} with base point $i$ is a trajectory starting and ending at the same vertex $i\in V$, of the form
\begin{equation*}
    l:i\xrightarrow{t_1} i_1\xrightarrow{t_2-t_1} i_2\xrightarrow{t_3-t_2}\dots \xrightarrow{t_{k-1}-t_{k-2}}i_{k-1}\xrightarrow{t_k-t_{k-1}} i\xrightarrow{t-t_k}
\end{equation*}
with $0<t_1<t_2<\dots<t_k<t$ and $i_k$ vertices of $V$, where the quantities on the arrows denote the time spent at the previous vertex before the jump to the next vertex. The last arrow means that we stay at $i$ for a time equal to $t-t_k$ without jumping and the trajectory ends. For example, $i\xrightarrow{t}$ is the trivial loop that makes no jump.

Given a Markov process with generator $L(i,j)$, we can define the bridge measure of the above trajectory with fixed $t>0$:
\begin{equation*}
    P_t^{(i,i)}(dl)=L(i,i_1)L(i_1,i_2)\dots L(i_{k-1},i) e^{-t_1L(i,i)-(t_2-t_1)L(i_1,i_1)-\dots-(t-t_k)L(i,i)}dt_1dt_2\dots dt_k.
\end{equation*}

The \emph{based loop measure} (associated to the generator $L$) is an infinite measure on the space of all admissible based loops, defined by
\begin{equation*}
    \mu(dl)=\sum_{i\in V}\int_{0}^{\infty}\frac{1}{t}P_t^{(i,i)}(dl)dt.
\end{equation*}
We can work on the equivalence class of based loops by forgetting the base point (via rerooting). The image of the loop measure $\mu$ under this equivalence relation is denoted $\mu^{\circ}$, and is called the (unbased) loop measure. For this paper, all computations are the same using either measure.

The Poissonian ensemble of loops $\mathcal{L}_{\alpha}$ for $\alpha>0$ is defined as the Poisson point process in the space of loops with intensity measure $\alpha\mu^{\circ}$. This means that, for any functional $F$ on the space of loops vanishing on loops with arbitrarily small length,
\begin{equation*}
    \mathbb{E}\left[\exp\left(i\sum_{l\in\mathcal{L}_{\alpha}}F(l)\right)\right]=\exp\left(\alpha\int\left(e^{iF(l)}-1\right)\mu^{\circ}(dl)\right)
\end{equation*}
where the integral on the right-hand side is over the space of all admissible unbased loops. Since the measures $\mu^{\circ}$ and $\mu$ coincide on loop functionals, equivalently one can use the intensity measure $\alpha\mu$ and the space of based loops in the above display.

The occupation time field of a (based or unbased) loop $l$ is
\begin{equation*}
    \widehat{l_i}=\int_{0}^{t}\mathbf{1}_{l(s)=i}ds,\quad \forall i\in V,
\end{equation*}
and the occupation time field of a collection of random loops $\mathcal{L}$ is $\widehat{\mathcal{L}}=\sum_{l\in\mathcal{L}}\widehat{l}$. In particular, the occupation field of the Poissonian loop soup  $\mathcal{L}_1$ with $\alpha=1$ will be denoted $\widehat{\mathcal{L}_1}$.

\subsection{Loop soup and Wilson's algorithm}
In the standard Markovian Wilson's algorithm, one can look at the resulting uniform spanning tree as well as the erased loops during the algorithm. There are deep symmetries involving these two random processes similar to the boson-fermion correspondence. Here we only collect some useful properties and refer to~\cite[Chapter~8]{MR4789605} for a detailed exposition.

Consider the standard Markovian Wilson's algorithms reviewed in Section~\ref{sec:ClassicalWilson}. The result that will be most relevant for us is the following identification~\cite[Corollary~8.1]{MR4789605}:

\begin{lemmA}[Occupation field of the random loops in the standard Markovian Wilson's algorithm]\label{lemm:OccupationFieldWilson_L1}
Consider a symmetric Markov jump process with generator $\widetilde{A}$ on the augmented graph $\widetilde{V}$ with killing at the root vertex $\delta$. The occupation time field defined by the random set of (based) erased loops during the standard Markovian Wilson's algorithm on $\widetilde{V}$ with root $\delta$ is independent of the random spanning tree and of the ordering of the vertices, and has the same distribution as the occupation time field $\widehat{\mathcal{L}_1}$ of the Poissonian loop soup on $V$ with generator $\widetilde{A}_{|V}$ of parameter $\alpha=1$.
\end{lemmA}

The standard Markovian Wilson's algorithm can be refined to give a complete reconstruction of the Poissonian loop soup $\mathcal{L}_1$ by applying a Poisson-Dirichlet decomposition of loops, see~\cite[Section~8.3]{MR4789605} for more information on this reconstruction. See also Section~\ref{subse:construction_of_general_reinforced_loop_soup_texorpdfstring_mathcal_l} for the construction of general Poissonian loop soup $\mathcal{L}_{\alpha}$ using Wilson's algorithms with any $\alpha>0$.

\subsection{Standard Markovian isomorphism theorems}\label{subse:standard Markovian_isomorphism_theorems}
We now recall a few standard Markovian isomorphism theorems, which are identities in law (i.e. exact relations on expectations) between the occupation time field and the (scalar) free field. A pedagogical reference to all the standard Markovian isomorphisms below is~\cite[Chapter~2]{MR2932978}.

\subsubsection{BFS-Dynkin isomorphism theorem}
Consider a standard Markov jump process $(Z_t)_{t\geq 0}$ of (symmetric) generator $\widetilde{A}$ starting at some vertex $a\in V$ on the augmented graph $\widetilde{V}$. Suppose that $(Z_t)_{t\geq 0}$ is killed at the root vertex $\delta$, denote by $\varrho$ the killing time, and let $\bm{S}=\bm{S}(\varrho)$ be the final local times (with $S_i=\int_{0}^{\varrho}\mathbf{1}_{\{Z_t=i\}}dt$).
\begin{theoA}[Standard Markovian BFS-Dynkin isomorphism theorem]\label{theoA:standard_BFSDynkin}
Let $\mathbb{E}^{\widetilde{A}}_{a}$ denote the expectation with respect to the process $Z$. For any smooth bounded function $g$ with rapid decay,
\begin{equation*}
    \mathbb{E}^{\widetilde{A}}_{a,b}\left[g(\bm{S}+\frac{1}{2}\bm{\phi}^2)\right]=\mathbb{E}\left[\phi_a\phi_b g(\frac{1}{2}\bm{\phi}^2)\right],
\end{equation*}
where $b$ is the vertex last visited by $Z$ before $\delta$, $\bm{\phi}$ is the Gaussian free field of generator $\widetilde{A}$ with pinning condition $\phi_{\delta}=0$, and $\mathbb{E}^{\widetilde{A}}_{a,b}[\bm{\cdot}]=\frac{1}{\widetilde{W}_{b\delta}}\mathbb{E}^{\widetilde{A}}_{a}[\bm{\cdot}\mathbf{1}_{\{Z_{\varrho^{-}}=b\}}]$.
\end{theoA}
See~\cite[Theorem~2.8]{MR2932978} and~\cite{MR902412,MR648362} for background and proof of this identity.

\subsubsection{Second generalized Ray-Knight isomorphism theorem}
Consider a standard Markov jump process $(Z_t)_{t\geq 0}$ of (symmetric) generator $A$ starting at some vertex $a\in V$ on the graph $V$ with local times $S_i=\int_{0}^{\varrho}\mathbf{1}_{\{Z_t=i\}}dt$. Denote by $\sigma(\gamma)=\inf\{t>0~;~S_a(t)>\gamma\}$ the first instant the local time at the starting point $a\in V$ exceeds $\gamma>0$, and $\bm{S}=\bm{S}(\sigma)$ the final local times.
\begin{theoA}[Standard Markovian second generalized Ray-Knight isomorphism theorem]\label{theoA:standard_RayKnight}
For any smooth bounded function $g$ with rapid decay,
\begin{equation*}
    \mathbb{E}^{A}_{a}\left[g(\bm{S}+\frac{1}{2}\bm{\phi}^2)\right]=\mathbb{E}\left[g\Big(\frac{1}{2}(\bm{\phi}+\sqrt{2\gamma})^2\Big)\right],
\end{equation*}
where $\bm{\phi}$ is the Gaussian free field with generator $A$ and pinning $\phi_{a}=0$.
\end{theoA}

We refer to~\cite{MR1813843} for background and proof of this identity. See also~\cite{MR3520013} for more connections to Vertex Reinforced Jump Process (reviewed in Section~\ref{subse:vertex_reinforced_jump_process}).

\subsubsection{Eisenbaum's isomorphism theorem}
Consider a standard Markov jump process $(Z_t)_{t\geq 0}$ of (symmetric) generator $\widetilde{A}$ starting at some vertex $a\in V$ on the augmented graph $\widetilde{V}$. Suppose that $(Z_t)_{t\geq 0}$ is killed at the root vertex $\delta$, denote by $\varrho$ the killing time, and let $\bm{S}=\bm{S}(\varrho)$ be the final local times (with $S_i=\int_{0}^{\varrho}\mathbf{1}_{\{Z_t=i\}}dt$).
\begin{theoA}[Standard Markovian Eisenbaum's isomorphism theorem]\label{theoA:standard_Eisenbaum}
For any smooth bounded function $g$ with rapid decay and any $s\in\mathbb{R}$,
\begin{equation*}
    \mathbb{E}^{A}_{a}\left[g\Big(\bm{S}+\frac{1}{2}(\bm{\phi}+s)^2\Big)\right]=\mathbb{E}\left[\frac{\phi_a+s}{\phi_\delta+s} g\Big(\frac{1}{2}(\bm{\phi}+s)^2\Big)\right],
\end{equation*}
where $\bm{\phi}$ is the Gaussian free field with generator $\widetilde{A}$ and pinning $\phi_{\delta}=0$.
\end{theoA}
We refer to~\cite[Thoerem~2.10]{MR2932978} and~\cite{MR1459468} for background and proof of this identity.

\subsubsection{Standard Markovian loop soup isomorphism}
Isomorphism theorems can be extended to the standard Markovian loop soup, which is somewhat already rooted in the original work of Dynkin~\cite{MR902412} and developed by Le Jan~\cite{MR2675000}. We will focus on the case with $\alpha=1$, although this identity can be generalized to any half integer-valued $\alpha$.

\begin{theoA}[Dynkin isomorphism for the loop soup]\label{lemm:Classical_Loop_Dynkin}
The occupation field of $\mathcal{L}_1$ of a loop soup with a symmetric generator $A$ has the same distribution as the average of the squares of two independent Gaussian free fields of the same generator $A$. In other words, for any smooth bounded function $g$ with rapid decay,
\begin{equation*}
    \mathbb{E}^{\text{loop},A}[g(\widehat{L_1})]=\int_{(\mathbb{R}^{|V|})^{2}}g\left(\frac{1}{2}(\bm{x}^2+\bm{y}^2)\right)e^{-\frac{1}{2}(\bm{x}A\bm{x}+\bm{y}A\bm{y})}\det(A)d\bm{x}d\bm{y}.
\end{equation*}
\end{theoA}

The proof of this theorem is due to Le Jan~\cite{MR2675000} can be found in~\cite[Theorem~4.5]{MR2932978} or~\cite[Theorem~6.1]{MR4789605}. In particular, in combination with Lemma~\ref{lemm:OccupationFieldWilson_L1}, this identity also applies to the occupation time field of the erased loops during the standard Markovian Wilson's algorithm on the augmented graph $\widetilde{V}$ with root $\delta$.

\subsection{Supersymmetric free field}\label{subse:supersymmetric_free_field}
We now recall some basic notions about supersymmetric field theories, following~\cite{wegner2016supermathematics} and~\cite{efetov1999supersymmetry}. We will be mainly focusing on the case of the so-called $(2,2)$-supersymmetric free field, and drop the specification $(2,2)$ when there is no ambiguity.

First, let us recall some basics of Grassmann calculus. Consider bosonic variables $(x,y,z,\dots)$ and fermionic variables $(\xi,\eta,\dots)$. Recall that when a variable is fermionic, it anticommutes with other fermionic variables and commutes with bosonic variables. For example,
\begin{equation*}
    \xi\eta=-\eta\xi,\quad \xi^2=\eta^2=0,\quad \xi x=x \xi,\quad \dots
\end{equation*}
A bosonic variable commutes with all variables. A bosonic variable is not necessarily real, e.g. $\xi\eta$ is bosonic. In general, one can deal with fermionic variables by representing a function with its formal (infinite) Taylor series in the fermionic variables: this series is a polynomial in the fermionic variables by anticommutation. For a systematic treatment of elementary Grassmann calculus towards applications in statistical physics, one can consult~\cite{wegner2016supermathematics,efetov1999supersymmetry}.

It is a standard fact that for Grassmann variables, derivation and integration are essentially the same operation, modulo a sign convention. We adopt the following convention throughout this article, that integration with the differential to the left coincides with left differentiation, e.g.
\begin{equation*}
    \int d\xi f(\xi)=\frac{\partial}{\partial\xi}f(\xi).
\end{equation*}
This is the convention used in~\cite{MR4255180,MR2728731}: see~\cite[Chapter~3]{wegner2016supermathematics} for a quick reminder.

Associate each vertex $i$ in the graph $V$ with a four vector $X_i=(x_i,y_i,\xi_i,\eta_i)$, with $x,y$ real variables and $\xi,\eta$ fermionic variables. We usually denote this by $X_i\in\mathbb{R}^{2|2}$, the superscript $2|2$ refers to the numbers of each type of variable. To define the supersymmetric free field with a language that is familiar to probabilists, one can think of the variables $X_i$ as spins, equipped with the (symmetric) inner product
\begin{equation}\label{eq:InnerProduct_SusyFreeField}
    X_i\cdot X_j=x_ix_j+y_iy_j+\xi_i\eta_j+\xi_j\eta_i,
\end{equation}
where the ordering of $\xi,\eta$ is important due to the anticommutation of fermionic variables. We sometimes drop the $\cdot$ and write this inner product as $X_iX_j$ when there is no ambiguity. An \emph{energy} term (equivalently, \emph{action functional} or \emph{Gibbs measure}) associated to each configuration $(X_i)_{i\in V}$ of this system induces (formally) a randomized spin configuration. The \emph{Berezin integral form} associated to the space $\mathbb{R}^{2|2}$ is
\begin{equation*}
    D\bm{X}=\prod_{i\in V}\frac{1}{2\pi}dx_idy_id\xi_id\eta_i,
\end{equation*}
which should be understood as the formal analog of the Lebesgue measure for $\mathbb{R}^2$.

At this stage, the setup is purely formal, and performing Grassmann integrations does not yield in general quantities that have probabilistic interpretations. However, special choices of the energy term yield interesting field theory for probabilists. A prime example is the supersymmetric free field, with the energy term defined as
\begin{equation*}
    \frac{1}{2}\bm{X}\Delta_W\bm{X}=\frac{1}{2}\sum_{i,j\in V}X_i(\Delta_W)_{ij}X_j=\sum_{(ij)\in E(V)}(X_i-X_j)W_{ij}(X_i-X_j).
\end{equation*}
As with the usual scalar free field, one needs some extra \emph{pinning} condition to make the integral $\int_{(\mathbb{R}^{2|2})^{|V|}}e^{-\frac{1}{2}\bm{X}\Delta_W\bm{X}}D\bm{X}$ converge.\footnote{The pinning condition on the augmented graph $\widetilde{V}$ is equivalent to a boundary condition on the graph $V$ with boundary $\{\delta\}$: we will use the pinning terminology in the sequel.} Therefore, consider the augmented graph $\widetilde{V}$ with extra vertex $\delta$ and its associated extra edge weights $(W_{i\delta})_{i\in V}$, with at least one positive $W_{i\delta}>0$ to preserve the connectedness of the graph. We usually choose the pinning condition with $\widetilde{X}_\delta=(0,0,0,0)=0$, and it is a special feature of the supersymmetric field theory (see Lemma~\ref{lemm:Parisi-Sourlas}) that the formal partition function is constant equal to $1$, i.e.
\begin{equation}\label{eq:TrivialParitionFunction}
    \int_{(\mathbb{R}^{2|2})^{|V|}}\mathbf{1}_{\{\widetilde{\bm{X}}_\delta=0\}}e^{-\frac{1}{2}\widetilde{\bm{X}}\Delta_{\widetilde{W}}\widetilde{\bm{X}}}D\widetilde{\bm{X}}=1.
\end{equation}
A systematic way to prove this is via the localization formula, see Corollary~\ref{coro:TrivialnessPartitionFunction} below.

The triviality of the partition function implies that there is no difference between the un-normalized and normalized expectations for the supersymmetric free field, and we denote them by $\llbracket\cdot\rrbracket_{\widetilde{W},\widetilde{X}_\delta=0}$. More specifically, for any smooth bounded function $F$ with rapid decay,
\begin{equation}\label{eq:susyFreeExpectation}
    \llbracket F(\widetilde{\bm{X}})\rrbracket_{\widetilde{W},\widetilde{X}_\delta=0}=\int_{(\mathbb{R}^{2|2})^{|V|}}F(\widetilde{\bm{X}})\mathbf{1}_{\{\widetilde{\bm{X}}_\delta=0\}}e^{-\frac{1}{2}\widetilde{\bm{X}}\Delta_{\widetilde{W}}\widetilde{\bm{X}}}D\widetilde{\bm{X}}
\end{equation}
is called the supersymmetric free field expectation with pinning at the root vertex (or the boundary) $\delta$. When all the edge weights $(W_{i\delta})_{i\in V}$ connecting to the root $\delta$ are constant equal to $h>0$, we also call this the supersymmetric free field expectation on $V$ with mass $h>0$.

\subsection{Supersymmetric hyperbolic sigma model}\label{subse:supersymmetric_hyperbolic_sigma_model}
We briefly recall next the definition of the $\text{H}^{2|2}$-model following~\cite{MR1134935,MR2728731} and its connections to the Vertex Reinforced Jump Process.

Consider a superfield $v_i=(x_i,y_i,z_i,\xi_i,\eta_i)\in\mathbb{R}^{3|2}$ defined on vertices $i\in V$, where the first three coordinates $(x_i,y_i,z_i)$ are real variables and the last two coordinates $(\xi_i,\eta_i)$ are fermionic variables. The (symmetric) inner product on this space is defined as
\begin{equation*}
    v_i\cdot v_j=x_ix_j+y_iy_j-z_iz_j+\xi_i\eta_j+\xi_j\eta_i.
\end{equation*}
This inner product is sometimes denoted by $v_iv_j$ when there is no ambiguity. We impose that this superfield $v_i$ lives on the supersymmetric hyperbolic space $\text{H}^{2|2}$ defined by the constraint that
\begin{equation*}
    \forall i\in V,\quad v_i^2=x_i^2+y_i^2-z_i^2+2\xi_i\eta_i=-1.
\end{equation*}
Under this constraint, by Taylor expansion of the square root at $1+x_i^2+y_i^2$ in the fermionic variables (see~\cite[Section~2]{MR2728731} if one is not familiar with this standard procedure in Grassmann calculus), one can express $z_i$ as a function of the other coordinates $(x_i,y_i,\xi_i,\eta_i)$. The convention is to choose the ``positive'' branch, namely
\begin{equation*}
    z_i=\sqrt{1+x_i^2+y_i^2}+\frac{\xi_i\eta_i}{\sqrt{1+x_i^2+y_i^2}}.
\end{equation*}

The origin or the zero-vector $0\in\text{H}^{2|2}$ is then defined to be $0=(0,0,1,0,0)\in\text{H}^{2|2}$. The terminology ``hyperbolic'' refers to the fact that if we forget about the fermionic variables, the constraint $x^2+y^2-z^2=-1$ is that of a standard hyperbolic model, and this appellation is best justified when one switches to the so-called horospherical coordinates using~\cite[Appendix~B]{MR2728731}. However for us and as pointed out in~\cite{MR4255180}, all our results are coordinate-free, and the horospherical coordinates will not be used in this paper for simplicity. We refer to~\cite{MR2728731} for discussion on the hyperbolic nature of the $\text{H}^{2|2}$-model, and continue with $\mathbb{R}^{3|2}$ coordinates.

Similarly to the supersymmetric free field theory above, we define the $\text{H}^{2|2}$ supersymmetric hyperbolic sigma model by defining its energy term (i.e. action functional) as
\begin{equation*}
    \frac{1}{2}\bm{v}\Delta_W\bm{v}=\sum_{i,j\in V}v_i (\Delta_{W})_{ij} v_j=-\sum_{(ij)\in E(V)}W_{ij}(x_ix_j+y_iy_j-z_iz_j+\xi_i\eta_j+\xi_j\eta_i+1),
\end{equation*}
and the Berezin integral form on the $\text{H}^{2|2}$-model is
\begin{equation*}
    D\mu(\bm{v})=\prod_{i\in V}\frac{1}{2\pi}\frac{dx_idy_i}{z_i}d\xi_id\eta_i.
\end{equation*}
Due to the non-compactness of the space $\text{H}^{2|2}$ equipped with the above measure, we renormalize the integral $\int_{(\text{H}^{2|2})^{|V|}}e^{-\frac{1}{2}\bm{v}\Delta_W\bm{v}}D\mu(\bm{v})$ by adding a pinning at an extra vertex $\delta$ on the augmented graph $\widetilde{V}$. The localization formula of supersymmetric field theory yields that the partition function with pinning at $\delta$ is again $1$ (see Corollary~\ref{coro:TrivialnessPartitionFunction}):
\begin{equation}\label{eq:TrivialPartitionFunction_H22}
    \int_{(\text{H}^{2|2})^{|V|}}\mathbf{1}_{\{\widetilde{v}_\delta=0\}}e^{-\frac{1}{2}\widetilde{\bm{v}}\Delta_{\widetilde{W}}\widetilde{\bm{v}}}D\mu(\widetilde{\bm{v}})=1,
\end{equation}
where $\widetilde{\bm{v}}=(v_i)_{i\in\widetilde{V}}$ is the augmented superfield on $\widetilde{V}=V\cup\{\delta\}$.

More generally, we denote by $\langle\cdot\rangle_{\widetilde{W},v_{\delta}=0}$ the formal expectation with respect to the above defined supersymmetric energy term: for smooth bounded functions $F$ with rapid decay,
\begin{equation}\label{eq:H22_expectation}
    \langle F(\widetilde{\bm{v}})\rangle_{\widetilde{W},v_{\delta}=0}=\int_{(\text{H}^{2|2})^{|V|}}F(\widetilde{\bm{v}})\mathbf{1}_{\{\widetilde{v}_\delta=0\}}e^{-\frac{1}{2}\widetilde{\bm{v}}\Delta_{\widetilde{W}}\widetilde{\bm{v}}}D\mu(\widetilde{\bm{v}}).
\end{equation}

\begin{rema}[``Equality in law'']\label{rema:equal_in_law}
We will often use some abuse of language by analogy to statistical physics. For example, the term ``expectation'' above is really just a (supersymmetric) integral, and one can realize the $\text{H}^{2|2}$-model via the supersymmetric free field using a ``change of measure''. All these terminologies should be understood in terms of the values of supersymmetric integrals against test functions: especially we will say two superfields are ``equal in law'' if for any smooth bounded function $F$ with rapid decay, the expectations obtained by supersymmetric integration in the form of~\eqref{eq:susyFreeExpectation} or~\eqref{eq:H22_expectation} of $F$ in these two superfields are equal.
\end{rema}

\subsection{Vertex Reinforced Jump Process}\label{subse:vertex_reinforced_jump_process}
We now introduce a random process with a reinforcement mechanism which is central to this paper: the Vertex Reinforced Jump Process. This process was already introduced by~\cite{MR1900324}, and has deep connections to the another random process with reinforcement called the edge reinforced random walk~\cite{MR1030727}. This connection was at the heart of the modern proofs of the famous Coppersmith-Diaconis' magic formula~\cite{coppersmith1987random}, and we recall some of the basic elements that are useful for this article.

Consider a finite connected graph $V$ with edge weights $W$. At time $s=0$, assign to each vertex $i\in V$ an initial positive \emph{local time} $\vartheta_i>0$. Specify also a vertex $i_0\in V$ (we stress that this is not necessarily the root vertex), and start a continuous time jump process $(Y_s)_{s\geq 0}$ starting from $i_0$, i.e. $Y_0=i_0$. The process $Y$ jumps from vertex $i$ to another vertex $j$ at time $s$ with rate $W_{ij}L_j(s)$, where $L_j(s)$ is the \emph{accumulated local time} (or simply local time) of the process $Y$ at time $s$, i.e.
\begin{equation*}
    L_j(s)=\vartheta_j+\int_{0}^{s}\mathbf{1}_{\{Y_\tau=j\}}d\tau.
\end{equation*}

This process is naturally not Markov, since the jump rate depends on the past. Most of the modern understanding of this process is achieved with the following time change~\cite{MR3420510,MR3531700}: consider the time scale
\begin{equation*}
    t=D(s)=\sum_{i\in V}(L_i(s)^2-\vartheta_i^2),
\end{equation*}
and consider the time-changed process $(Z_t)_{t\geq 0}$ defined by
\begin{equation*}
    Z_t=Y_{D^{-1}(t)}
\end{equation*}
which has its own local times defined as
\begin{equation}\label{eq:changed_local_times}
    S_i(t)=\int_{0}^{t}\mathbf{1}_{\{Z_\tau=i\}}d\tau=L_i(D^{-1}(t))^2-\vartheta_i^2.
\end{equation}
for each vertex $i\in V$.

We will see below in Section~\ref{subse:connections_between_vertex_reinforced_jump_processes_and_the_supersymmetric_hyperbolic_sigma_model} that the time-changed process $(Z_t)_{t\geq 0}$ acquires the so-called partial exchangeability property~\cite{MR786142}, thus is a mixture of standard Markovian jump processes in some random environment. This is an important philosophy throughout this article.

\subsection{Connections between Vertex Reinforced Jump Processes and the supersymmetric hyperbolic sigma model}\label{subse:connections_between_vertex_reinforced_jump_processes_and_the_supersymmetric_hyperbolic_sigma_model}
The $\text{H}^{2|2}$-model gains an extra probabilistic interpretation when looking at some marginal laws of the bosonic variables. The connection was unveiled by Sabot-Tarrès~\cite{MR3420510}, and we only recall here some essential results that we use in this article.

\begin{lemmA}[Mixture of Markov jump processes]\label{lemm:mixutre_of_markov_jump_processes}
Recall the time-changed process $(Z_t)_{t\geq 0}$ in Section~\ref{subse:vertex_reinforced_jump_process} starting from the vertex $i_0\in V$ with initial local times $\bm{\vartheta}$. It can be sampled in the following way. First, sample the environment $\bm{u}=\{u_i\in \mathbb{R}, u_{i_0}=0\}_{i\in V\setminus\{i_0\}}$ according to the following probability density function on $\mathbb{R}^{|V|-1}$ with prescribed value $u_{i_0}=0$:
\begin{equation}\label{eq:MixingDensityU_i0}
\begin{split}
    &d\nu_{i_0}^{W,\bm{\vartheta}}(\bm{u})\\
    ={}&\mathbf{1}_{\{u_{i_0}=0\}}e^{-\frac{1}{2}\sum_{(ij)\in E(V)}W_{ij}(e^{u_i-u_j}\vartheta_j^2+e^{u_j-u_i}\vartheta_i^2-2\vartheta_i\vartheta_j)}\sqrt{D(W,\bm{u})}\prod_{i\in V, i\neq i_0}\frac{\vartheta_i e^{-u_i}du_i}{\sqrt{2\pi}}
\end{split}
\end{equation}
where the expression $D(W,\bm{u})$ is a sum over the spanning trees $T$ of the graph,
\begin{equation}\label{eq:Determinant_MatrixTree}
    D(W,\bm{u})=\sum_{T}\prod_{(ij)\in E(T)}W_{ij}e^{u_i+u_j}.
\end{equation}
Then, given a sample of the environment $\bm{u}$, sample the Markov jump process with (static or time-independent) jump rate from a vertex $i\in V$ to another vertex $j\in V$:
\begin{equation*}
    \frac{1}{2}W_{ij}e^{u_j-u_i}.
\end{equation*}

The law of $(Z_t)_{t\geq 0}$ is the same as the resulting continuous time jump process. We say that $(Z_t)_{t\geq 0}$ is a mixture of standard Markovian jump processes in the random environment~\eqref{eq:MixingDensityU_i0}.
\end{lemmA}

We also say that the above static Markov jump process is the \emph{quenched} process (in the environment $\bm{u}$). The fact that~\eqref{eq:MixingDensityU_i0} is a probability density function is highly non-trivial without this lemma: a direct proof can be found in~\cite{letac2017multivariate} and a supersymmetric proof in~\cite[Equation~(5.1)]{MR2728731}. The term $D(W,\bm{u})$ is actually a determinant (see the discussion after~\cite[Theorem~2]{MR3420510}), and~\eqref{eq:Determinant_MatrixTree} is obtained by the famous matrix tree theorem.

As a consequence of the previous lemma, we get the following generalization to the Vertex Reinforced Jump Process on the augmented graph $\widetilde{V}$ killed at the root vertex $\delta$. All the notations are similarly defined on the augmented graph, and the mixing measure becomes
\begin{equation}\label{eq:killed_MixingDensityU_i0}
\begin{split}
    &d\nu_{i_0}^{\widetilde{W},\bm{\widetilde{\vartheta}}}(\widetilde{\bm{u}})\\
    ={}&\mathbf{1}_{\{u_{i_0}=0\}}e^{-\frac{1}{2}\sum_{(ij)\in E(\widetilde{V})}\widetilde{W}_{ij}(e^{u_i-u_j}\vartheta_j^2+e^{u_j-u_i}\vartheta_i^2-2\vartheta_i\vartheta_j)}\sqrt{D(\widetilde{W},\widetilde{\bm{u}})}\prod_{i\in \widetilde{V}, i\neq i_0}\frac{\vartheta_i e^{-u_i}du_i}{\sqrt{2\pi}}.
\end{split}
\end{equation}
The law of $(\widetilde{Z}_t)_{t\geq 0}$ is also that of a mixture of standard Markovian jump processes with jump rate\begin{equation*}
    \frac{1}{2}\widetilde{W}_{ij}e^{u_j-u_i}
\end{equation*}
starting at $i_0\in V$ and killed at the root $\delta\in\widetilde{V}$ in the random environment defined by~\eqref{eq:killed_MixingDensityU_i0}.

\begin{rema}
Notice that $\bm{u}$ above is a vector indexed by $V\setminus\{i_0\}$ and $\widetilde{\bm{u}}$ is a vector indexed by $V$. This is a slight abuse of notations due to the extra pinning condition.
\end{rema}

We record a ``change of starting point'' formula in the sequel for the mixing measure.
\begin{lemm}[Shift in the mixing measure]\label{eq:change_of_starting_point_mixing_measure}
Let $a,b\in V$ and consider the mixing measures $d\nu^{W,\bm{1}}_a$ and $d\nu^{W,\bm{1}}_b$ with different starting points. If $\bm{u}'$ and $\bm{u}$ are such that $u'_i=u_i-u_b$ for all $i\in V$, then
\begin{equation*}
    e^{u_b-u_a}d\nu^{W,\bm{1}}_a(\bm{u})=d\nu^{W,\bm{1}}_b(\bm{u}').
\end{equation*}
\end{lemm}
As a consequence, if $F$ is a function of the gradients $(u_i-u_j)_{i\in V}$, then
\begin{equation*}
    \int_{\mathbb{R}^{|V|-1}}F((u_i-u_j)_{i\in V})e^{u_b-u_a}d\nu^{W,\bm{1}}_a(\bm{u})=\int_{\mathbb{R}^{|V|-1}}F((u'_i-u'_j)_{i\in V})d\nu^{W,\bm{1}}_a(\bm{u}').
\end{equation*}

\begin{proof}
This follows directly from the definition of the mixing measure~\eqref{eq:MixingDensityU_i0}.
\end{proof}

The second result is the surprising link between the Vertex Reinforced Jump Process and the $\text{H}^{2|2}$ supersymmetric field theory.
\begin{theoA}[Marginal law of $\text{H}^{2|2}$ and the mixing measure of the Vertex Reinforced Jump Process]\label{lemm:Link_VRJP_H22}
Consider the following change of coordinates: for each $i\in\widetilde{V}$, define
\begin{equation*}
\begin{split}
    &x_i=\sinh(t_i)-(\frac{1}{2}s_i^2+\psi_i\overline{\psi}_i)e^{t_i},~y_i=s_ie^{t_i},~z_i=\cosh(t_i)+(\frac{1}{2}s_i^2+\psi_i\overline{\psi}_i)e^{t_i},\\
    &\xi_i=\psi_i e^{t_i},~\eta_i=\overline{\psi}_i e^{t_i}.
\end{split}
\end{equation*}
Then $(\widetilde{\bm{t}},\widetilde{\bm{s}},\widetilde{\overline{\bm{\psi}}},\widetilde{\bm{\psi}})$ is the horospherical coordinates of the $\text{H}^{2|2}$-model. The marginal law of $\widetilde{\bm{t}}$ of the $\text{H}^{2|2}$-model defined by~\eqref{eq:H22_expectation} is equal to the law of the random environment $\widetilde{\bm{u}}$ defined with the mixing measure~\eqref{eq:killed_MixingDensityU_i0} and $i_0=\delta$ (in the sense of Remark~\ref{rema:equal_in_law}).
\end{theoA}
This result is best understood when one uses the horospherical coordinates representation of the $\text{H}^{2|2}$-model as in~\cite{MR2728731}. Since we do not use this particular geometric input in the sequel, we refer to~\cite{MR2728731} and~\cite{MR3420510} for the background and proof of this fundamental result. We do not directly use Theorem~\ref{lemm:Link_VRJP_H22}, but this philosophy will be important in the (self-contained) proof of the supersymmetric Bayes formula, Theorem~\ref{th:susy_Bayes} below.

\subsection{Localization formula}\label{subse:localization_formula}
In supersymmetric field theories, a fundamental formula is the so-called (supersymmetric) localization formula. We will use especially the following formulation:
\begin{lemmA}[Parisi-Sourlas formula]\label{lemm:Parisi-Sourlas}
Recall the supersymmetric free field $\bm{X}$ defined in Section~\ref{subse:supersymmetric_free_field}. Let $F((X_iX_j)_{i,j\in V})$ be a smooth bounded function on the inner products $(X_iX_j)_{i,j\in V}$ with rapid decay. Then for the supersymmetric free field expectation with pinning at $\delta\in\widetilde{V}$, we have
\begin{equation*}
    \llbracket F \rrbracket_{\widetilde{W},\widetilde{X}_{\delta}=0}=\int_{(\mathbb{R}^{2|2})^{|V|}}F((X_iX_j)_{i,j\in V})e^{-\frac{1}{2}\bm{X}(\Delta_W+H)\bm{X}}D\widetilde{X}=F(0),
\end{equation*}
where $H$ is the diagonal matrix with diagonal coefficients $(W_{i\delta})_{i\in V}$.
\end{lemmA}
Many proofs of this lemma can be found in the literature, e.g.~\cite[Lemma~16]{MR2728731}. An important consequence is the following~\cite[Proposition~2]{MR2728731}:
\begin{coroA}[Trivialness of the partition functions]\label{coro:TrivialnessPartitionFunction}
The partitions functions for the supersymmetric free field~\eqref{eq:TrivialParitionFunction} and for the $\text{H}^{2|2}$-model~\eqref{eq:TrivialPartitionFunction_H22} are both equal to $1$, i.e.
\begin{equation*}
    \llbracket 1 \rrbracket_{\widetilde{W},\widetilde{X}_{\delta}=0}=\langle 1 \rangle_{\widetilde{W},\widetilde{\Phi}_{\delta}=0}=1.
\end{equation*}
\end{coroA}
\begin{proof}
Taking $F\equiv 1$ in the Parisi-Sourlas formula, we get the trivialness of the partition function for the supersymmetric free field~\eqref{eq:TrivialParitionFunction}. For the $\text{H}^{2|2}$-model the argument is similar: define $z_i=\sqrt{1+X_i^2}$ for $i\in V$ (which is a function of the inner products) and let
\begin{equation*}
    F=\mathbf{1}_{\{\widetilde{\bm{v}}_\delta=0\}}e^{-\frac{1}{2}\widetilde{\bm{z}}\Delta_{\widetilde{W}}\widetilde{\bm{z}}-\sum_{i\in V}W_{i\delta}(z_i-1)}\prod_{i\in V}\frac{1}{z_i}
\end{equation*}
be a rapidly decaying function on the inner products $X_iX_j$. Plugging this in Lemma~\ref{lemm:Parisi-Sourlas} yields the triviality of the partition function for the $\text{H}^{2|2}$-model.
\end{proof}

Another consequence of the localization formula is the invariance of the inner product $v_i\cdot v_j$ and the Berezin integral form $D\mu(v)$ under the Lorentz boost
\begin{equation}\label{eq:LorentzBoost}
    \theta_s(x,y,z,\xi,\eta)=(x\cosh s+z\sinh s,y,z\cosh s+x\sinh s,\xi,\eta).
\end{equation}
We refer to~\cite[Appendix~B]{MR2728731} for a proof. Notice that the Lorentz invariance is broken for the normalized supersymmetric expectation $\langle\cdot\rangle_{W,\delta}$ with pinning at the root vertex $\delta$: indeed, the Lorentz invariance also acts on the boundary condition, so that on the augmented graph $\widetilde{V}$,
\begin{equation*}
    \forall s\in\mathbb{R},\quad \langle F(\widetilde{\bm{v}})\rangle_{\widetilde{W},v_\delta=\theta_s(0)}=\int_{(\text{H}^{2|2})^{|V|}}F(\widetilde{\bm{v}})\mathbf{1}_{\{v_{\delta}=\theta_s(0)\}}e^{-\frac{1}{2}\widetilde{\bm{v}}\Delta_{\widetilde{W}}\widetilde{\bm{v}}}D\mu(\widetilde{\bm{v}}),
\end{equation*}
and similarly for expressions on the graph $V$.

\section{Wilson's algorithm with reinforcement}\label{sec:wilson_s_algorithm_with_reinforcement}
Consider a connected finite graph $V$ with a distinguished vertex $\delta$ and edge weights $W_{ij}$ for $i,j\in V$ (we write $V$ instead of $\widetilde{V}$ for rooted graphs in sections dealing with Wilson's algorithms). We describe below how an infinite chain of Vertex Reinforced Jump Process starting from $\delta$ can be used to generate a pair formed by a random spanning tree and a random collection of loops on the graph $V$. This process will be called \emph{Wilson's algorithm with reinforcement}, or the \emph{reinforced Wilson's algorithm}. We further show that, under certain random initial conditions, the pair of the random spanning tree and the random collection of loops on $V$ obtained by Wilson's algorithm with reinforcement can be obtained using a mixture of the standard Markovian Wilson's algorithm in some random environment.

\subsection{Wilson's algorithm generated by a single reinforced vertex jump process}\label{subse:WilsonWithOneExploration}
Consider a Vertex Reinforced Jump Process on a connected finite graph $V$ starting from a distinguished vertex $\delta$. Suppose that the initial local time of each vertex $i$ in $V$ is $\vartheta_i$ and run a Vertex Reinforced Jump Process starting from $\delta$, without being killed. It is known that this process is almost surely recurrent~\cite{MR3420510}, therefore eventually visits every vertex in $V$ infinitely many times.

We now describe our definition of Wilson's algorithm with reinforcement, whose outcome is a pair formed by a random spanning tree and a random collection of loops on the graph $V$. The main difference with the standard Markovian Wilson's algorithm is that the random walk will switch between two forms: a \emph{``visible''} form where we perform loop erasure to construct the spanning tree and the loop soup, and an \emph{``invisible''} form where we stop the loop erasure but only accumulate local times for the reinforcement mechanism.\footnote{We choose the word ``visible'' to be coherent with terminologies in~\cite[Chapter~4]{MR3616205} where the cycle-popping construction of Wilson's algorithm is reviewed.} The reinforced Wilson's algorithm gradually constructs a spanning tree as we will now explain.

\begin{enumerate}
    \item We start with the single-vertex graph $\mathcal{T}=\{\delta\}$: vertices of $\mathcal{T}$ will be called inactive. Define $\mathcal{A}$ to be the complement of vertices of $\mathcal{T}$ in $V$: this is the collection of active vertices.
    \item We run a Vertex Reinforced Jump Process $Y$ starting from the vertex $\delta$ with edge weights $(W_{ij})_{i,j\in V}$ and initial local times $\{\vartheta_i\}_{i\in V}$, and we denote by $\mathcal{O}=\emptyset$ the collection of loops at this stage.
    \item The process is in its invisible form when it is on a vertex inside $\mathcal{T}$. When it jumps to another vertex inside $\mathcal{T}$, we only update the local times according to the rules of the Vertex Reinforced Jump Process.
    \item The first time the process jumps to a vertex in $\mathcal{A}$, declare the process to be visible. When the process is visible, it updates the local times according to the rules of the Vertex Reinforced Jump Process. Furthermore, for the visible part of the process, we perform the loop erasure procedure as in the standard Markovian Wilson's algorithm until it jumps to a vertex in $\mathcal{T}$.
    \item Upon returning to a vertex in $\mathcal{T}$, the process is back to its invisible form, and we update $\mathcal{T}$ and the collection of loops $\mathcal{O}$ in the following way: the erased loops during the active form of the random walks will be added to the collection of loops $\mathcal{O}$, and the remaining self-avoiding branch leading to $\mathcal{T}$ will be added to form an enlarged $\mathcal{T}$.
    \item We repeat the steps $(3)-(5)$ until the moment when $\mathcal{T}$ is a spanning tree of $V$. From this moment, the process will always be in its invisible form, so that $\mathcal{T}$ and $\mathcal{O}$ will not change.
\end{enumerate}

In the sequel, we call the final random spanning tree $\mathbb{T}^{\bm{\vartheta}}$ (as the final state of $\mathcal{T}$) the \emph{reinforced spanning tree} and the final occupation time field $\widehat{\mathbb{L}^{\bm{\vartheta}}_1}$ of the random loop process (as the final state of $\mathcal{O}$) the \emph{reinforced loop soup occupation field}, with initial local time conditions $\{\vartheta_i\}_{i\in V}$.

\begin{rema}
We only need the occupation time field $\widehat{\mathbb{L}^{\bm{\vartheta}}_1}$ of the reinforced loop soup in our main Theorem~\ref{th:SUSY_BFSD_Loop}, and the reconstruction of the reinforced loop soup $\mathbb{L}^{\bm{\vartheta}}_1$ from the reinforced Wilson's algorithm is postponed to Section~\ref{subse:construction_of_general_reinforced_loop_soup_texorpdfstring_mathcal_l}.
\end{rema}

\subsection{The standard Markovian Wilson's algorithm generated by a single infinite Markov jump process}
We first study the analog of the above-defined algorithm in the standard Markovian setting, i.e. without the reinforcement mechanism.
\begin{lemm}[Standard Markovian Wilson's algorithm using one infinite exploration]\label{lemm:standard_oneshot_Wilson}
In the setting of a connected finite graph $V$ with the usual notations, a variant of Wilson's algorithm can be realized using the following procedure:
\begin{enumerate}
    \item We start with the single-vertex graph $\mathcal{T}=\{\delta\}$: vertices of $\mathcal{T}$ will be called inactive. Define $\mathcal{A}$ to be the complement of vertices of $\mathcal{T}$ in $V$: this is the collection of active vertices.
    \item We run a Markov chain $X$ starting from the vertex $\delta$, and we denote by $\mathcal{O}=\emptyset$ the collection of loops at this stage.
    \item The process is in its invisible form when it is on a vertex inside $\mathcal{I}$.
    \item The first time the process jumps to a vertex in $\mathcal{A}$, declare the process to be visible. For the visible part of the process, we perform the loop erasure procedure until it jumps to a vertex in $\mathcal{T}$.
    \item Upon returning to a vertex in $\mathcal{T}$, the process is back to its invisible form, and we update the collection of loops $\mathcal{L}$ in the following way: the erased loops during the visible form of the random walks will be added to the collection of loops $\mathcal{O}$, and the remaining branch leading to $\mathcal{T}$ will be added to form an enlarged $\mathcal{T}$.
    \item We reiterate the steps $(3)-(5)$ until the moment $\mathcal{T}$ is a spanning tree of $V$. From this moment, the process will always be in its invisible form, so that $\mathcal{T}$ and $\mathcal{O}$ will not change.
\end{enumerate}
Then the final random spanning tree $\mathcal{T}$ and the occupation time field of the final random loop process $\mathcal{O}$ have the same joint distribution as the uniform spanning tree and the occupation time field of the random loop process constructed from the standard Markovian Wilson's algorithm of Section~\ref{sec:ClassicalWilson}. In particular, the occupation time field of the final $\mathcal{O}$ is equal in law to $\widehat{\mathcal{L}_1}$, the occupation field of the Poissonian loop soup with $\alpha=1$.
\end{lemm}
Notice that this is the algorithm in the above section except that no reinforcement mechanism is implemented: in particular no data is updated during the process.
\begin{proof}
It suffices to check that this algorithm effectively pops all the possible cycles in the cycle-popping form of the standard Markovian Wilson's algorithm in Section~\ref{subse:wilson_s_algorithm_and_cycle_popping}. Since the final random spanning tree $\mathcal{T}$ and the occupation time field of the random loop collection $\mathcal{O}$ are independent of the order in which the cycles are popped~\cite[Lemma~4.2]{MR3616205}, they have the same joint distribution as the pair described in Lemma~\ref{lemm:classicalWilson_cyclePopping} or Lemma~\ref{lemm:classicalWilson_loopErasure}.
\end{proof}

\subsection{Reinforced Wilson's algorithm as a mixture of standard Markovian Wilson's algorithms in random environment}
This proof of the following lemma is essentially the same as the proof of Lemma~\ref{lemm:mixutre_of_markov_jump_processes} originally stated for Vertex Reinforced Jump Processes.
\begin{lemm}[Mixture of standard Markovian Wilson's algorithms in random environment]\label{lemm:mixture_of_Markov_Wilson}
The randomized reinforced Wilson's algorithm is a mixture of standard Markovian Wilson's algorithm for a Markov chain in a random environment $\bm{u}$, with jump rates
\begin{equation*}
    \frac{1}{2}W_{ij}e^{u_j-u_i},
\end{equation*}
and the law of the random field $\bm{u}$ is given by the mixing measure $d\nu_{\delta}^{W,\bm{\vartheta}}(\bm{u})$ defined in~\eqref{eq:MixingDensityU_i0}.
\end{lemm}
\begin{proof}
By Lemma~\ref{lemm:mixutre_of_markov_jump_processes}, the Vertex Reinforced Jump Process is a mixture of standard Markovian jump process with mixing measure $d\nu_{\delta}^{W,\bm{\vartheta}}(\bm{u})$ defined in~\eqref{eq:MixingDensityU_i0} and the above jump rate. Conditioning on the environment $\bm{u}$ and performing the standard Wilson's algorithm, which outputs the quenched uniform spanning tree and random loop process. In particular, the quenched Wilson's algorithm can be performed using the form of Lemma~\ref{lemm:standard_oneshot_Wilson}, as all the standard Markovian Wilson's algorithms yield the same uniform tree and random loop process in law. Since our reinforced Wilson's algorithm of Section~\ref{subse:WilsonWithOneExploration} is the same as that of Lemma~\ref{lemm:standard_oneshot_Wilson} in the quenched sense, this finishes the proof that the reinforced Wilson's algorithm of Section~\ref{subse:WilsonWithOneExploration} is the annealed version of Lemma~\ref{lemm:standard_oneshot_Wilson}, integrated in the random environment $\bm{u}$ with the mixing measure $d\nu_{\delta}^{W,\bm{\vartheta}}(\bm{u})$.
\end{proof}

The rules we defined for the reinforced Wilson's algorithm in Section~\ref{subse:WilsonWithOneExploration} are such that the process is partially exchangeable in the sense of~\cite{MR786142}, which guarantees the mixture of Markov process representation of Lemma~\ref{lemm:mixture_of_Markov_Wilson}.

\section{A supersymmetric Bayes formula}\label{sec:a_supersymmetric_bayes_formula}
In this section, we introduce a supersymmetric Bayes formula which will be used multiple times in the sequel. We first introduce some shorthand notations.

Recall from Section~\ref{subse:connections_between_vertex_reinforced_jump_processes_and_the_supersymmetric_hyperbolic_sigma_model} that the Vertex Reinforced Jump Process on $V$ can be realized by first sampling a random environment $\bm{u}$ on the graph $V$ with mixing measure~\eqref{eq:MixingDensityU_i0}, then sample the standard Markovian jump process with the following jump rate from a vertex $i\in V$ to another vertex $j\in V$:
\begin{equation*}
    \frac{1}{2}W_{ij}e^{u_j-u_i}.
\end{equation*}
This motivates the following notation: introduce the square matrix $A^{\bm{u}}$ of size $|V|$ with entries
\begin{equation}\label{eq:Definition_AuMatrix}
    A^{\bm{u}}_{ij}=\begin{cases} -W_{ij} & i\neq j \\ \sum_{k\in V\setminus\{i\}}W_{ik}e^{u_k-u_i} & i=j \end{cases}.
\end{equation}
Then the generator of the quenched process, that is the standard Markovian jump process in the environment $\bm{u}$, has generator $\frac{1}{2}e^{-{\bm{u}}}A^{\bm{u}}e^{\bm{u}}$, where $e^{\bm{u}}$ is the diagonal matrix of size $|V|$ with diagonal entries $(e^{u_i})_{i\in V}$. Similarly, on the augmented graph with root vertex $\delta$, the mixing measure of the random environment $\widetilde{\bm{u}}$ induces the following square matrix $\widetilde{A^{\bm{u}}}$ of size $|\widetilde{V}|$ with entries
\begin{equation}\label{eq:Definition_AuMatrix_tilded}
    \widetilde{A^{\bm{u}}_{ij}}=\begin{cases} -W_{ij} & i\neq j \\ \sum_{k\in \widetilde{V}\setminus\{i\}}W_{ik}e^{u_k-u_i} & i=j \end{cases}.
\end{equation}

The matrix $A^{\bm{u}}$ (resp. $\widetilde{A^{\bm{u}}}$) is not symmetric, but we can introduce the following symmetric square matrix $B^{\bm{u}}=e^{\bm{u}}A^{\bm{u}}e^{\bm{u}}$ (resp. $\widetilde{B^{\bm{u}}}=e^{\widetilde{\bm{u}}}\widetilde{A^{\bm{u}}}e^{\widetilde{\bm{u}}}$). Now it makes sense to speak about the supersymmetric free field in Section~\ref{subse:supersymmetric_free_field} with the matrix $W$ replaced by the matrix $B^{\bm{u}}$ (resp. $\widetilde{B^{\bm{u}}}$), for which the (unpinned) supersymmetric free field expectation will be denoted $\llbracket \cdot \rrbracket_{B^{\bm{u}}}$ (resp. $\llbracket \cdot \rrbracket_{\widetilde{B^{\bm{u}}}}$) as in Section~\ref{subse:supersymmetric_free_field}. Now we introduce the shorthand notations:
\begin{equation}\label{eq:ShortHand_Au_Expectation}
    \llbracket F(X)\rrbracket_{A^{\bm{u}}}=\llbracket F(e^{\bm{u}}X)\rrbracket_{B^{\bm{u}}},\quad \llbracket F(\widetilde{X})\rrbracket_{\widetilde{A^{\bm{u}}}}=\llbracket F(e^{\widetilde{\bm{u}}}\widetilde{X})\rrbracket_{\widetilde{B^{\bm{u}}}}.
\end{equation}

Finally, recall the Lorentz boost $\theta_s$ with $s\in\mathbb{R}$ defined in Section~\ref{subse:localization_formula}: in particular, $\theta_s(0)=(\sinh s, 0, \cosh s, 0,0)$ for the $(2,2)$-supersymmetric free field and $\theta_s(0)=(\sinh s, 0, 0,0)$ for the $\text{H}^{2|2}$-model. We are now ready to announce the supersymmetric Bayes formula.\footnote{See~\cite{MR2278358} on this terminology, which we borrow here by analogy without using any Bayesian analysis.}
\begin{theo}[Supersymmetric Bayes formula]\label{th:susy_Bayes}
Recall that $\llbracket \cdot \rrbracket$ denotes the supersymmetric free field expectation defined in Section~\ref{subse:supersymmetric_free_field} and above, $\langle\cdot\rangle$ denotes the $\text{H}^{2|2}$ expectation defined in Section~\ref{subse:supersymmetric_hyperbolic_sigma_model}, and $\nu^{W,\bm{z}}_{a}$ denotes the mixing measure defined in Section~\ref{subse:connections_between_vertex_reinforced_jump_processes_and_the_supersymmetric_hyperbolic_sigma_model}. Then the following $\text{H}^{2|2}$ expectations can be realized as supersymmetric free field expectations in random environments: for any smooth bounded function $g$ with rapid decay,
\begin{equation*}
\begin{split}
    &\left\langle\frac{z_a}{z_{\delta}}\int_{\mathbb{R}^{|V|}}g(\widetilde{\bm{x}},\widetilde{\bm{y}},\widetilde{\bm{\xi}},\widetilde{\bm{\eta}},\widetilde{\bm{u}})d\nu^{\widetilde{W},\widetilde{\bm{z}}}_{a}(\widetilde{\bm{u}})\right\rangle_{\widetilde{W},v_{\delta}=\theta_s(0)}\\
    ={}&\int_{\mathbb{R}^{|V|}}\llbracket g(\widetilde{\bm{x}},\widetilde{\bm{y}},\widetilde{\bm{\xi}},\widetilde{\bm{\eta}},\widetilde{\bm{u}})\rrbracket_{\widetilde{A^{\bm{u}}},X_{\delta}=\theta_s(0)} d\nu^{\widetilde{W},\bm{1}}_{a}(\widetilde{\bm{u}}),\\
    &\left\langle\int_{\mathbb{R}^{|V|-1}}g(\bm{x},\bm{y},\bm{\xi},\bm{\eta},\bm{u})d\nu^{W,\bm{z}}_{b}(\bm{u})\right\rangle_{W,v_{\delta}=\theta_s(0)}\\
    ={}&\int_{\mathbb{R}^{|V|-1}}\llbracket g(\bm{x},\bm{y},\bm{\xi},\bm{\eta},\bm{u})\rrbracket_{A^{\bm{u}},X_{\delta}=\theta_s(0)} d\nu^{W,\bm{1}}_{b}(\bm{u}),
\end{split}
\end{equation*}
for any $a\in\widetilde{V},b\in V$ and $s\in\mathbb{R}$.
\end{theo}
Notice that on the right-hand sides in the above displays, the mixing measures are defined with constant initial local times $\bm{1}$, while on the left-hand sides, the local times are random and sampled according to the law of the $z$-coordinate given by the $\text{H}^{2|2}$-model.

\begin{proof}
We prove the first identity and the second one follows similarly. The proof goes by direct computation, which slightly generalizes the proof of the result in~\cite{MR3420510} recalled in Section~\ref{subse:connections_between_vertex_reinforced_jump_processes_and_the_supersymmetric_hyperbolic_sigma_model}. We exploit the cancellations between the action functional of the $\text{H}^{2|2}$-model and the mixing measure of the random environment associated to a vertex jump reinforced process.

First, we expand the supersymmetric expectation $\langle \cdot\rangle$. Denote by $F(\widetilde{\bullet})=\int_{\mathbb{R}^{|V|}}g(\widetilde{\bullet},\widetilde{\bm{u}})d\nu_a^{\widetilde{W},z}(\widetilde{\bm{u}})$ and by the definition of the $\text{H}^{2|2}$ expectation~\eqref{eq:H22_expectation},
\begin{equation*}
\begin{split}
    \left\langle\frac{z_a}{z_{\delta}}\int_{\mathbb{R}^{|V|}}g(\widetilde{\bullet},\widetilde{\bm{u}})d\nu^{\widetilde{W},\widetilde{\bm{z}}}_{a}(\widetilde{\bm{u}})\right\rangle_{\widetilde{W},v_{\delta}=\theta_s(0)}&=\left\langle\frac{z_a}{z_{\delta}}F(\widetilde{\bullet})\right\rangle_{\widetilde{W},v_{\delta}=\theta_s(0)}\\
    &=\int_{(\text{H}^{2|2})^{|V|}}\frac{z_a}{z_{\delta}}F(\widetilde{\bullet})\mathbf{1}_{\{v_\delta=\theta_s(0)\}}e^{-\frac{1}{2}\widetilde{\bm{v}}\Delta_{\widetilde{W}}\widetilde{\bm{v}}}D\mu(\widetilde{\bm{v}}).
\end{split}
\end{equation*}
\begin{small}
Further expanding out the energy term and the Berezin integral form, we get
\begin{equation*}
    \int_{(\text{H}^{2|2})^{|V|}}F(\widetilde{\bullet})\mathbf{1}_{\{v_\delta=\theta_s(0)\}}e^{\sum\limits_{(ij)\in E(\widetilde{V})}W_{ij}(x_ix_j+y_iy_j-z_iz_j+\xi_i\eta_j+\xi_j\eta_i+1)}\prod_{i\in\widetilde{V}\setminus\{a\}}\frac{1}{z_i}\prod_{i\in\widetilde{V}\setminus\{\delta\}}\frac{1}{2\pi}dx_idy_id\xi_id\eta_i.
\end{equation*}
\end{small}

Next, we expand the random environment $d\nu_a^{\widetilde{W},\widetilde{\bm{z}}}(\widetilde{\bm{u}})$, i.e. we write
\begin{small}
\begin{equation*}
    F(\widetilde{\bullet})=\int_{\mathbb{R}^{|V|}}g(\widetilde{\bullet},\widetilde{\bm{u}})\mathbf{1}_{\{u_a=0\}}e^{-\frac{1}{2}\sum\limits_{(ij)\in E(\widetilde{V})}W_{ij}(e^{u_j-u_i}z_i^2+e^{u_i-u_j}z_j^2-2z_iz_j)}\sqrt{D(\widetilde{W},\widetilde{\bm{u}})}\prod_{i\in\widetilde{V}\setminus\{a\}}\frac{z_ie^{-u_i}}{\sqrt{2\pi}}du_i.
\end{equation*}
\end{small}

We now plug the above equality into the penultimate display. Let us first work out the exponent in the exponential: recalling that $z_i^2=x_i^2+y_i^2+2\xi_i\eta_i+1$ by the constraint in the $\text{H}^{2|2}$ coordinates,
\begin{small}
\begin{equation*}
\begin{split}
    &\sum\limits_{(ij)\in E(\widetilde{V})}W_{ij}(x_ix_j+y_iy_j-z_iz_j+\xi_i\eta_j+\xi_j\eta_i+1)-\frac{1}{2}\sum\limits_{(ij)\in E(\widetilde{V})}W_{ij}(e^{u_j-u_i}z_i^2+e^{u_i-u_j}z_j^2-2z_iz_j)\\
    ={}&\sum\limits_{(ij)\in E(\widetilde{V})}W_{ij}(x_ix_j+y_iy_j+\xi_i\eta_j+\xi_j\eta_i+1)-\frac{1}{2}\sum\limits_{(ij)\in E(\widetilde{V})}W_{ij}(e^{u_j-u_i}z_i^2+e^{u_i-u_j}z_j^2)\\
    ={}&-\frac{1}{2}\sum_{(ij)\in E(\widetilde{V})}W_{ij}(e^{u_j-u_i}+e^{u_i-u_j}-2)\\
    &\quad\quad+\sum\limits_{(ij)\in E(\widetilde{V})}W_{ij}(x_ix_j+y_iy_j+\xi_i\eta_j+\xi_j\eta_i)-\frac{1}{2}\sum\limits_{i\in \widetilde{V}}\sum_{k\in\widetilde{V}\setminus\{i\}}W_{ik}e^{u_k-u_i}(x_i^2+y_i^2+2\xi_i\eta_i)\\
    ={}&-\frac{1}{2}\sum_{(ij)\in E(\widetilde{V})}W_{ij}(e^{u_j-u_i}+e^{u_i-u_j}-2)-\frac{1}{2}\widetilde{\bm{X}}\widetilde{A^{\bm{u}}}\widetilde{\bm{X}},
\end{split}
\end{equation*}
\end{small}
where in the first equality, we cancelled the $z_iz_j$ terms; in the second equality, we used the $\text{H}^{2|2}$ constraint $z_i^2=x_i^2+y_i^2+2\xi_i\eta_i+1$; in the third equality, we used the definitions of $X_i=(x_i,y_i,\xi_i,\eta_i)$ with the inner product~\eqref{eq:InnerProduct_SusyFreeField} and $\widetilde{A^{\bm{u}}}$ the matrix defined in~\eqref{eq:Definition_AuMatrix_tilded}.

In short, there is no $z$-component anymore in the exponent in the exponential. Finally,
\begin{equation*}
\begin{split}
    &\left\langle\frac{z_a}{z_{\delta}}\int_{\mathbb{R}^{|V|}}g(\widetilde{\bullet},\widetilde{\bm{u}})d\nu^{\widetilde{W},\bm{z}}_{a}(\widetilde{\bm{u}})\right\rangle_{\widetilde{W},v_{\delta}=\theta_s(0)}\\
    ={}&\int_{(\text{H}^{2|2})^{|V|}}\int_{\mathbb{R}^{|V|}}g(\widetilde{\bullet},\widetilde{\bm{u}})\mathbf{1}_{\{X_{\delta}=\theta_s(0)\}}e^{-\frac{1}{2}\widetilde{\bm{X}}\widetilde{A^{\bm{u}}}\widetilde{\bm{X}}}\prod_{i\in V}\frac{dx_idy_id\xi_id\eta_i}{2\pi}d\nu_{a}^{\widetilde{W},\bm{1}}(\widetilde{\bm{u}})\\
    ={}&\int_{\mathbb{R}^{|V|}}\llbracket g(\widetilde{\bullet},\widetilde{\bm{u}})\rrbracket_{\widetilde{A^{\bm{u}}},X_{\delta}=\theta_s(0)} d\nu^{\widetilde{W},\bm{1}}_{a}(\widetilde{\bm{u}})
\end{split}
\end{equation*}
where the first equality follows from putting all the $u$-dependent terms together to get the mixing measure $d\nu_{a}^{\widetilde{W},\bm{1}}(\widetilde{\bm{u}})$, and the second quality from the discussion at the beginning of this section.
\end{proof}

\section{Isomorphism theorems for the Vertex Reinforced Jump Process}\label{sec:isomorphism_theorems_for_the_vertex_reinforced_jump_process}
We now turn to the isomorphism theorems for the Vertex Reinforced Jump Process, originally discovered in~\cite{MR4021254,MR4255180}. Our formulation is slightly different only due to the different choices of notations but they are essentially equivalent. The main novelty here is that, following the ideas of~\cite{MR3420510}, we demonstrate that these isomorphism theorems can be obtained simply by integrating the standard Markovian isomorphism theorems in a well-chosen random environment generated by the Vertex Reinforced Jump Process. In other words, all theorems below are \emph{annealed} versions of the corresponding standard Markovian isomorphisms recalled in Section~\ref{subse:standard Markovian_isomorphism_theorems}.

We will be using two slightly different expectations below:
\begin{enumerate}
    \item The notation $\mathbb{E}^{\widetilde{W},\widetilde{\bm{\vartheta}}}_{a}$ is associated with a Vertex Reinforced Jump Process on the augmented graph $\widetilde{V}$ (sometimes assumed killed at the root vertex $\delta$), with edge weights $\widetilde{W}$ and initial local times $\widetilde{\bm{\vartheta}}$. For example, if $\widetilde{\bm{\vartheta}}=\bm{1}$, then we start the reinforced process with all initial local times constant equal to $1$.
    \item The notation $\mathbb{E}^{\widetilde{\bm{u}}}_{a}$ is associated with a standard Markovian continuous time jump process on the augmented graph $\widetilde{V}$ (sometimes assumed killed at the root vertex $\delta$) in the environment defined by $\widetilde{\bm{u}}$, with jump rate from $i$ to $j$ equal to $\frac{1}{2}\widetilde{W}_{ij}e^{u_j-u_i}$ as in Lemma~\ref{lemm:mixutre_of_markov_jump_processes} and the discussions after. We omit the edge weights $\widetilde{W}$ in this notation.
\end{enumerate}

\subsection{BFS-Dynkin isomorphism theorem}
The BFS-Dynkin isomorphism theorem for the Vertex Reinforced Jump Process and the $\text{H}^{2|2}$-model is the following identity:
\begin{theo}[Reinforced BFS-Dynkin isomorphism theorem]\label{th:Reinforced_BFS-Dynkin}
Let $(Y_s)_{s\geq 0}$ be a Vertex Reinforced Jump Process on the augmented graph starting at some vertex $a\in V$ and killed when it hits the root vertex $\delta$. Let $\widetilde{\rho}$ be the killing time, and $b=Y_{\widetilde{\rho}^{-}}\in V$ be the vertex $Y$ last visits before $\delta$. Let $\widetilde{\bm{L}}$ be the final local times, i.e. $\widetilde{L}_i=1+\int_{0}^{\widetilde{\rho}}\mathbf{1}_{\{Y_s=i\}}ds$ for all $i\in \widetilde{V}$. Then for any smooth bounded function $g$ with rapid decay,
\begin{equation*}
    \int_{0}^{\infty}\mathbb{E}^{\widetilde{W},\bm{1}}_{a}[g(\widetilde{\bm{L}})\mathbf{1}_{\{Y_{s}=b,s<\widetilde{\rho}\}}]ds=\langle x_ax_bg(\widetilde{\bm{z}})\rangle_{\widetilde{W},v_{\delta}=0},
\end{equation*}
where $\langle\cdot\rangle_{\widetilde{W},v_{\delta}=0}$ is the $\text{H}^{2|2}$ expectation with pinning at $\delta$ defined in Section~\ref{subse:supersymmetric_hyperbolic_sigma_model}.

In particular, with the uniform pinning condition $W_{\delta,i}=h>0$ for all $i\in V$,
\begin{equation*}
    \mathbb{E}^{\widetilde{W},\bm{1}}_{a}[g(\widetilde{\bm{L}})\mathbf{1}_{\{Y_{\widetilde{\rho}^-}=b\}}]=h\langle x_ax_bg(\widetilde{\bm{z}})\rangle_{\widetilde{W},v_{\delta}=0}.
\end{equation*}
\end{theo}

We prove this theorem by integrating the standard Markovian BFS-Dynkin isomorphism theorem, Theorem~\ref{theoA:standard_BFSDynkin}. We first need a slightly more general supersymmetric version of the standard Markovian BFS-Dynkin isomorphism theorem. The idea of having such supersymmetric forms of the isomorphism theorems can be traced back to Le Jan and Luttinger~\cite{MR941982,MR713539}.
\begin{lemm}[Supersymmetric BFS-Dynkin isomorphism theorem]\label{lemm:ClassicalSusy_BFSD}
Let $Z$ be a standard Markovian jump process on the augmented graph with rate $\frac{1}{2}\widetilde{W}_{ij}e^{u_j-u_i}$ (for some fixed environment $\widetilde{\bm{u}}$) starting at some vertex $a\in V$ and killed at time $\widetilde{\varrho}$ when it hits the root vertex $\delta$. Let $b=Z_{\widetilde{\varrho}^{-}}\in V$ be the vertex $Z$ last visits before it hits $\delta$. Let $\widetilde{\bm{S}}$ be the final local times $S_i=\int_{0}^{\widetilde{\varrho}}\mathbf{1}_{\{Z_t=i\}}dt$. Then with $z_i^2=x_i^2+y_i^2+2\xi_i\eta_i+1$ and any smooth bounded $g$ with rapid decay,
\begin{equation*}
    \left\llbracket\mathbb{E}^{\widetilde{\bm{u}}}_{a}[g(\widetilde{\bm{S}}+\widetilde{\bm{z}}^2)\mathbf{1}_{\{Z_{\widetilde{\varrho}^{-}}=b\}}]\right\rrbracket_{\widetilde{A^{\bm{u}}},X_{\delta}=0}=\widetilde{W}_{b\delta}\left\llbracket e^{u_\delta-u_a}x_ax_b g(\widetilde{\bm{z}}^2)\right\rrbracket_{\widetilde{A^{\bm{u}}},X_{\delta}=0}.
\end{equation*}
where $X=(x_i,y_i,\xi_i,\eta_i)$ be the supersymmetric free field and $\llbracket\cdot \rrbracket_{\widetilde{A^{\bm{u}}},X_{\delta}=0}$ is defined in Section~\ref{sec:a_supersymmetric_bayes_formula}.
\end{lemm}
\begin{proof}
This is done by rewriting some determinants of the standard Markovian BFS-Dynkin isomorphism theorem using Grassmann Gaussian integrals instead of real Gaussian integrals. Namely, recall that for a symmetric real definite positive matrix $B$ of size $|V|$, we have
\begin{equation*}
\begin{split}
    \int_{\mathbb{R}^{|V|}}e^{-\frac{1}{2}\bm{x}B\bm{x}}\prod_{i\in V}\frac{dx_i}{\sqrt{2\pi}}&=\det(B)^{-\frac{1}{2}},\\
    \int e^{-\bm{\xi}B\bm{\eta}}\prod_{i\in V}d\xi_{i}d\eta_{i}&=\det(B).
\end{split}
\end{equation*}

Now recall the standard Markovian BFS-Dynkin isomorphism, following~\cite[Chapter~6]{MR4789605} or~\cite[Chapter~2]{MR2932978}. Consider a standard continuous time Markov jump process $Z$ on the augmented graph with generator $A$, starting at $a\in V$ and killed at the root vertex $\delta$ at time $\varrho$. Let $\widetilde{\bm{S}}$ be the final local times at time $\varrho$, and let $b\in V$ be the last vertex visited by $Z$ before $\delta$. Denote by $H$ the diagonal matrix of size $|V|$ with diagonal entries $(W_{i\delta})_{i\in V}$. The classical BFS-Dynkin isomorphism can be written as an equality between two Laplace transforms with parameter $\bm{\lambda}$:
\begin{equation*}
\begin{split}
    &\int_{\mathbb{R}^{|V|}}\frac{1}{\widetilde{W}_{b\delta}}\mathbb{E}^{A}_{a}[e^{-\langle\bm{\lambda},\bm{S}\rangle}e^{-\frac{1}{2}\langle\bm{\lambda},\bm{x}^2\rangle}\mathbf{1}_{\{Z_{\varrho^{-}}=b\}}]e^{-\frac{1}{2}\bm{x}A\bm{x}}\prod_{i\in V}\frac{dx_i}{\sqrt{2\pi}}\\
    ={}&\int_{\mathbb{R}^{|V|}} x_ax_b e^{-\frac{1}{2}\langle\bm{\lambda},\bm{x}^2\rangle}e^{-\frac{1}{2}\bm{x}A\bm{x}}\prod_{i\in V}\frac{dx_i}{\sqrt{2\pi}}
\end{split}
\end{equation*}
where $\mathbb{E}^{A}_{a}$ denotes the expectation with respect to the process $Z$ starting at $a$ with the relation $\mathbb{E}^{A}_{a}[\bm{\cdot}]=\frac{1}{\widetilde{W}_{b\delta}}\mathbb{E}^{A}_{a,b}[\bm{\cdot}\mathbf{1}_{\{Z_{\varrho^{-}}=b\}}]$. Our discussion above implies equality between a new pair of Laplace transforms with $z_i^2=x_i^2+y_i^2+2\xi_i\eta_i+1$:
\begin{equation*}
\begin{split}
    &\int \frac{1}{\widetilde{W}_{b\delta}}\mathbb{E}^{A}_{a}[e^{-\langle\bm{\lambda},\bm{S}\rangle}e^{-\frac{1}{2}\langle\bm{\lambda},\bm{z}^2\rangle}\mathbf{1}_{\{Z_{\varrho^{-}}=b\}}]e^{-\frac{1}{2}\bm{x}A\bm{x}}e^{-\frac{1}{2}\bm{y}A\bm{y}-\bm{\xi}A\bm{\eta}}\prod_{i\in V}\frac{dx_idy_i}{2\pi}d\xi_{i}d\eta_{i}\\
    ={}&\int x_ax_b e^{-\frac{1}{2}\langle\bm{\lambda},\bm{z}^2\rangle}e^{-\frac{1}{2}\bm{x}A\bm{x}}e^{-\frac{1}{2}\bm{y}A\bm{y}-\bm{\xi}A\bm{\eta}}\prod_{i\in V}\frac{dx_idy_i}{2\pi}d\xi_{i}d\eta_{i},
\end{split}
\end{equation*}
from which follows Lemma~\ref{lemm:ClassicalSusy_BFSD} by considering the generator $\frac{1}{2}e^{-\widetilde{\bm{u}}}\widetilde{A^{\bm{u}}}e^{\widetilde{\bm{u}}}$ instead of $A$ as in the beginning of Section~\ref{sec:a_supersymmetric_bayes_formula}. Since we used some non-conventional notations (only the matrix $\widetilde{B^{\bm{u}}}=e^{\widetilde{\bm{u}}}\widetilde{A^{\bm{u}}}e^{\widetilde{\bm{u}}}$ is symmetric thus can be used to define simultaneously a Markov process and a Gaussian free field), we explain this last step with some further details (which is basically a change of variables and does not involve any supersymmetric calculation).

Start by applying the previous discussion with $A=\widetilde{B^{\bm{u}}}$, which yields
\begin{equation*}
    \left\llbracket\mathbb{E}^{\widetilde{\bm{u}}}_{a}\left[g\left(\widetilde{\bm{S}}+\frac{1}{2}\widetilde{\bm{z}}^2\right)\mathbf{1}_{\{Z_{\widetilde{\varrho}^{-}}=b\}}\right]\right\rrbracket_{\widetilde{B^{\bm{u}}},X_{\delta}=0}=\widetilde{W}_{b\delta}e^{u_\delta+u_b}\left\llbracket x_ax_b g\left(\frac{1}{2}\widetilde{\bm{z}}^2\right)\right\rrbracket_{\widetilde{B^{\bm{u}}},X_{\delta}=0}.
\end{equation*}

Recall from~\eqref{eq:Definition_AuMatrix_tilded} that $\llbracket F(X)\rrbracket_{\widetilde{A^{\bm{u}}}}=\llbracket F(e^{\widetilde{\bm{u}}}X)\rrbracket_{\widetilde{B^{\bm{u}}}}$, and the Markov process with generator $\frac{1}{2}e^{-\widetilde{\bm{u}}}\widetilde{A^{\bm{u}}}e^{\widetilde{\bm{u}}}$ and $\widetilde{B^{\bm{u}}}$ only differs by a factor $\frac{1}{2}e^{-2\widetilde{\bm{u}}}$, which means that they are equivalent up to a change of local time scale, i.e.
\begin{equation}\label{eq:Relation_AB_localtimes}
    S^{(\widetilde{B^{\bm{u}}})}_i=\frac{1}{2}e^{-2u_i}\cdot S^{(\frac{1}{2}e^{-\widetilde{\bm{u}}}\widetilde{A^{\bm{u}}}e^{\widetilde{\bm{u}}})}_i.
\end{equation}
Combining we get
\begin{equation*}
\begin{split}
    &\left\llbracket\mathbb{E}^{\widetilde{\bm{u}}}_{a}\left[g\left(\frac{1}{2}e^{-2\widetilde{\bm{u}}}\widetilde{\bm{S}}+\frac{1}{2}e^{-2\widetilde{\bm{u}}}\widetilde{\bm{z}}^2\right)\mathbf{1}_{\{Z_{\widetilde{\varrho}^{-}}=b\}}\right]\right\rrbracket_{\widetilde{A^{\bm{u}}},X_{\delta}=0}\\
    ={}&\widetilde{W}_{b\delta}e^{u_\delta+u_b}\left\llbracket e^{-u_a-u_b}x_ax_b g\left(\frac{1}{2}e^{-2\widetilde{\bm{u}}}\widetilde{\bm{z}}^2\right)\right\rrbracket_{\widetilde{A^{\bm{u}}},X_{\delta}=0}.
\end{split}
\end{equation*}
Rescaling the function $g$ by the factor $\frac{1}{2}e^{-2\widetilde{\bm{u}}}$ concludes the proof of the lemma.
\end{proof}

\begin{proof}[Proof of Theorem~\ref{th:Reinforced_BFS-Dynkin}]
Recall the relation between the original Vertex Reinforced Jump Process $(Y_s)_{s\geq 0}$ and the associated time-changed process $(Z_t)_{t\geq 0}$ described in Section~\ref{subse:vertex_reinforced_jump_process}. Denote by $\widetilde{\rho}$ the killing time for $Y$ and $\widetilde{\varrho}$ the killing time for $Z$. With $S_i=D(L_i)$ the final local times for the process $Z$ (where $D$ is the time-change function in~\eqref{eq:changed_local_times} of Section~\ref{subse:vertex_reinforced_jump_process}), we have
\begin{equation*}
\begin{split}
    \mathbb{E}^{\widetilde{W},\bm{1}}_{a}[g(\widetilde{\bm{L}})\mathbf{1}_{\{Y_{\widetilde{\rho}^{-}}=b\}}]&=\left\langle z_a\mathbb{E}^{\widetilde{W},\widetilde{\bm{z}}}_a[g(\widetilde{\bm{L}})\mathbf{1}_{\{Y_{\widetilde{\rho}^{-}}=b\}}]\right\rangle_{\widetilde{W},v_{\delta}=0}\\
    &=\left\langle z_a\int_{\mathbb{R}^{|V|}}\mathbb{E}^{\widetilde{\bm{u}}}_a[g(\sqrt{\widetilde{\bm{S}}+\widetilde{\bm{z}}^2})\mathbf{1}_{\{Z_{\widetilde{\varrho}^{-}}=b\}}]d\nu^{\widetilde{W},\widetilde{\bm{z}}}_a(\widetilde{\bm{u}})\right\rangle_{\widetilde{W},v_{\delta}=0}\\
    &=\int_{\mathbb{R}^{|V|}}\left\llbracket\mathbb{E}^{\widetilde{\bm{u}}}_{a}[g(\sqrt{\widetilde{\bm{S}}+\widetilde{\bm{z}}^2})\mathbf{1}_{\{Z_{\widetilde{\varrho}^{-}}=b\}}]\right\rrbracket_{\widetilde{A^{\bm{u}}},X_{\delta}=0} d\nu^{\widetilde{W},\bm{1}}_a(\widetilde{\bm{u}}),
\end{split}
\end{equation*}
where the first equality follows from the localization formula of Lemma~\ref{lemm:Parisi-Sourlas}, the second equality from the random environment representation of the $\text{H}^{2|2}$ expectation recalled in Section~\ref{subse:connections_between_vertex_reinforced_jump_processes_and_the_supersymmetric_hyperbolic_sigma_model}, and the third equality from the supersymmetric Bayes formula of Theorem~\ref{th:susy_Bayes}.

Now, when the environment $\widetilde{\bm{u}}$ is fixed, with the above Lemma~\ref{lemm:ClassicalSusy_BFSD} for the standard Markovian supersymmetric isomorphism theorem, we can further write
\begin{equation*}
\begin{split}
    \mathbb{E}^{\widetilde{W},\bm{1}}_{a}[g(\widetilde{\bm{L}})\mathbf{1}_{\{Y_{\widetilde{\rho}^{-}}=b\}}]&=\widetilde{W}_{b\delta}\int_{\mathbb{R}^{|V|}}\left\llbracket e^{u_\delta-u_a} x_ax_b g(\widetilde{\bm{z}})\right\rrbracket_{\widetilde{A^{\bm{u}}},X_{\delta}=0} d\nu^{\widetilde{W},\bm{1}}_a(\widetilde{\bm{u}})\\
    &=\widetilde{W}_{b\delta}\int_{\mathbb{R}^{|V|}}\left\llbracket x_ax_b g(\widetilde{\bm{z}})\right\rrbracket_{\widetilde{A^{\bm{u}}},X_{\delta}=0} d\nu^{\widetilde{W},\bm{1}}_\delta(\widetilde{\bm{u}})\\
    &=\widetilde{W}_{b\delta}\left\langle x_ax_b g(\widetilde{\bm{z}})\int_{\mathbb{R}^{|V|}}d\nu^{\widetilde{W},\widetilde{\bm{z}}}_\delta(\widetilde{\bm{u}})\right\rangle_{\widetilde{W},v_{\delta}=0}\\
    &=\widetilde{W}_{b\delta}\langle x_ax_b g(\widetilde{\bm{z}})\rangle_{\widetilde{W},v_{\delta}=0},
\end{split}
\end{equation*}
where the second equality follows from the shift equation~\eqref{eq:change_of_starting_point_mixing_measure}, the third equality uses again the supersymmetric Bayes formula of Theorem~\ref{th:susy_Bayes}, and the last equality from that the mixing measure $d\nu^{\widetilde{W},\widetilde{\bm{z}}}_\delta(\widetilde{\bm{u}})$ integrates to $1$.

We finish the proof by integrating the above identity over $t$. The case with uniform pinning $\widetilde{W}_{\delta i}=h$ for all $i\in V$ follows from that $\widetilde{\rho}$ is an independent exponential random variable of parameter $h>0$.
\end{proof}

\subsection{Generalized second Ray-Knight theorem}
The same philosophy applies to the generalized second Ray-Knight theorem for the Vertex Reinforced Jump Process. The same steps as in the proof of Theorem~\ref{th:Reinforced_BFS-Dynkin} apply.
\begin{theo}[Reinforced generalized second Ray-Knight theorem]\label{theo:Reinforced_RayKnight}
Consider a Vertex Reinforced Jump Process $(Y_s)_{s\geq 0}$ starting at some vertex $a$ on the graph $V$ (without killing) and initial local times $\bm{1}$. Let $\gamma>1$ and let $\tau(\gamma)=\inf\{t>0~;~L_a(t)>\gamma\}$ be the first instant when the local time at $a$ exceeds $\gamma$. Then for any $s\in\mathbb{R}$ and any smooth bounded function $g$ with rapid decay,
\begin{equation*}
    \mathbb{E}^{W,\bm{1}}_{a}[g(\bm{L}(\tau(\cosh(s))))]=\langle g(\theta_s(\bm{z}))\rangle_{W,v_{a}=0}
\end{equation*}
where the right hand side is the $\text{H}^{2|2}$ expectation with pinning condition at $a$ and $\theta_s$ is the Lorentz boost~\eqref{eq:LorentzBoost}.
\end{theo}

As a preliminary, we need the following supersymmetric free field version of the generalized second Ray-Knight theorem:
\begin{lemm}[Supersymmetric free field generalized second Ray-Knight theorem]
Consider a Markov jump process $Z$ starting at vertex $a$ on the graph $V$ with generator $\frac{1}{2}e^{-\bm{u}}A^{\bm{u}}e^{\bm{u}}$, where $\bm{u}$ is some fixed environment and $A^{\bm{u}}$ is the matrix~\eqref{eq:Definition_AuMatrix}. Let $\gamma>0$ and $\sigma(\gamma)=\inf\{t>0~;~S_a(t)>\gamma\}$, where $\bm{S}$ denotes the local times (with initial local times $\bm{0}$). Then for any $s\in\mathbb{R}$ and any smooth bounded function $g$ with rapid decay,
\begin{equation*}
    \llbracket \mathbb{E}^{\bm{u}}_a[g(\bm{S}(\sigma(\sinh^2(s)))+\bm{z}^2)]\rrbracket_{A^{\bm{u}},X_{a}=0}=\llbracket g(\bm{z}^2)\rrbracket_{A^{\bm{u}},X_{a}=\theta_s(0)},
\end{equation*}
where $\llbracket\cdot\rrbracket_{A^{\bm{u}},X_{a}=\star}$ is the expectation with respect to the supersymmetric free field $(X_i)_{i\in V}=((x_i,y_i,\xi_i,\eta_i))_{i\in V}$ twisted by $A^{\bm{u}}$ with pinning condition (defined at the beginning of Section~\ref{sec:a_supersymmetric_bayes_formula}) and $z_i^2=x_i^2+y_i^2+2\xi_i\eta_i+1$ for all $i\in V$.
\end{lemm}
The proof is similar to that of Lemma~\ref{lemm:ClassicalSusy_BFSD} (by rewriting determinants in the original generalized second Ray-Knight theorem (see Theorem~\ref{theoA:standard_RayKnight}) arising from the Gaussian free field as fermionic integration) and is omitted.

\begin{proof}[Proof of Theorem~\ref{theo:Reinforced_RayKnight}]
The proof is very similar to that of Theorem~\ref{th:Reinforced_BFS-Dynkin}.

Start with the supersymmetric localization formula,
\begin{equation*}
    \mathbb{E}^{W,\bm{1}}_{a}[g(\bm{L}(\tau(\cosh(s))))]=\left\langle\mathbb{E}^{W,\bm{z}}_{a}[g(\bm{L}(\tau(\cosh(s))))]\right\rangle_{W,v_{a}=0}
\end{equation*}
then the relation $L_i(\tau(\cosh(s)))=\sqrt{S_i(\sigma(\sinh^2(s)))+z_i^2}$,
\begin{equation*}
    \left\langle\mathbb{E}^{W,\bm{z}}_{a}[g(\bm{L}(\tau(\cosh(s))))]\right\rangle_{W,v_{a}=0}=\left\langle\mathbb{E}^{W,\bm{z}}_{a}[g(\sqrt{\bm{S}(\sigma(\sinh^2(s)))+\bm{z}^2})]\right\rangle_{W,v_{a}=0}
\end{equation*}
and the interpretation of the Vertex Reinforced Jump Process as a standard Markovian jump process in random environment:
\begin{equation*}
\begin{split}
    &\left\langle\mathbb{E}^{W,\bm{z}}_{a}[g(\sqrt{\bm{S}(\sigma(\sinh^2(s)))+\bm{z}^2})]\right\rangle_{W,v_{a}=0}\\
    ={}&\left\langle\int_{\mathbb{R}^{|V|-1}}\mathbb{E}^{W,\bm{z}}_{a}[g(\sqrt{\bm{S}(\sigma(\sinh^2(s)))+\bm{z}^2})]d\nu^{W,\bm{z}}_{a}(\bm{u})\right\rangle_{W,v_{a}=0}.
\end{split}
\end{equation*}

From this we apply the supersymmetric Bayes formula of Theorem~\ref{th:susy_Bayes},
\begin{equation*}
\begin{split}
    &\left\langle\int_{\mathbb{R}^{|V|-1}}\mathbb{E}^{W,\bm{z}}_{a}[g(\sqrt{\bm{S}(\sigma(\sinh^2(s)))+\bm{z}^2})]d\nu^{W,\bm{z}}_{a}(\bm{u})\right\rangle_{W,v_{a}=0}\\
    ={}&\int_{\mathbb{R}^{|V|-1}}\left\llbracket\mathbb{E}^{W,\bm{z}}_{a}[g(\sqrt{\bm{S}(\sigma(\sinh^2(s)))+\bm{z}^2})]\right\rrbracket_{A^{\bm{u}},X_{a}=0}d\nu^{W,\bm{1}}_{a}(\bm{u}),
\end{split}
\end{equation*}
the supersymmetric free field version of the generalized second Ray-Knight theorem above,
\begin{equation*}
\begin{split}
    &\int_{\mathbb{R}^{|V|-1}}\left\llbracket\mathbb{E}^{W,\bm{z}}_{a}[g(\sqrt{\bm{S}(\sigma(\sinh^2(s)))+\bm{z}^2})]\right\rrbracket_{A^{\bm{u}},X_{a}=0}d\nu^{W,\bm{1}}_{a}(\bm{u})\\
    ={}&\int_{\mathbb{R}^{|V|-1}}\llbracket g(\bm{z})\rrbracket_{A^{\bm{u}},X_{a}=\theta_s(0)}d\nu^{W,\bm{1}}_{a}(\bm{u}),
\end{split}
\end{equation*}
and supersymmetric Bayes back to the $\text{H}^{2|2}$ expectation using again Theorem~\ref{th:susy_Bayes},
\begin{equation*}
    \int_{\mathbb{R}^{|V|-1}}\llbracket g(\bm{z})\rrbracket_{A^{\bm{u}},X_{a}=\theta_s(0)}d\nu^{W,\bm{1}}_{a}(\bm{u})=\left\langle g(\bm{z})\int_{\mathbb{R}^{|V|-1}}d\nu^{W,\bm{z}}_{a}(\bm{u})\right\rangle_{A^{\bm{u}},X_{a}=\theta_s(0)},
\end{equation*}
and as $d\nu^{W,\bm{z}}_{a}(\bm{u})$ integrates to $1$ for each $\bm{z}$,
\begin{equation*}
    \left\langle g(\bm{z})\int_{\mathbb{R}^{|V|-1}}d\nu^{W,\bm{z}}_{a}(\bm{u})\right\rangle_{A^{\bm{u}},X_{a}=\theta_s(0)}=\langle g(\bm{z})\rangle_{A^{\bm{u}},X_{a}=\theta_s(0)}=\langle g(\theta_s(\bm{z}))\rangle_{A^{\bm{u}},X_{a}=0}
\end{equation*}
where we finish with the Lorentz boost invariance~\eqref{eq:LorentzBoost}.
\end{proof}

\subsection{Eisenbaum's isomorphism theorem}
Finally, we obtain a reinforced version of Eisenbaum's isomorphism theorem following the same lines. Using the same notations as in the reinforced BFS-Dynkin isomorphism, Theorem~\ref{th:Reinforced_BFS-Dynkin}, the reinforced Eisenbaum's isomorphism theorem takes the following form:
\begin{theo}[Reinforced Eisenbaum's isomorphism theorem]\label{theo:Reinforced_Eisenbaum}
Consider a Vertex Reinforced Jump Process $(Y_s)_{s\geq 0}$ on the augmented graph $\widetilde{V}$ with initial local times $\bm{z}$, starting at some vertex $a\in V$ and killed at the root vertex $\delta$ at time $\rho$. If $\widetilde{\bm{L}}$ denotes the final local times at instant $\rho$, then for any $s\in\mathbb{R}$ and any smooth bounded function $g$ with rapid decay,
\begin{equation*}
    \left\langle \frac{z_a}{z_\delta}\mathbb{E}^{\widetilde{W},\widetilde{\bm{z}}}_a[g(\widetilde{\bm{L}})]\right\rangle_{\widetilde{W},v_{\delta}=\theta_s(0)}=\left\langle \frac{x_a}{x_{\delta}}g(\widetilde{\bm{z}})\right\rangle_{\widetilde{W},v_{\delta}=\theta_s(0)},
\end{equation*}
where the expectation on both sides are the $\text{H}^{2|2}$ expectation with pinning condition $v_{\delta}=\theta_s(0)$ at the root vertex $\delta$ and $\theta_s$ is the Lorentz boost~\eqref{eq:LorentzBoost}.
\end{theo}

The theorem above is written to be equivalent to the statement in~\cite{MR4255180}. It seems healthy to record the simpler form of the above theorem with initial local times $\bm{1}$:
\begin{coro}[Reinforced Eisenbaum's isomorphism theorem]
With the same notations as in the previous theorem, we have
\begin{equation*}
    \mathbb{E}^{\widetilde{W},\bm{1}}_a[g(\widetilde{\bm{L}})]=\left\langle \frac{x_a}{x_{\delta}}g(\widetilde{\bm{z}})\right\rangle_{\widetilde{W},v_{\delta}=0}.
\end{equation*}
\end{coro}
\begin{proof}
We only need to observe that when $s=0$ (that is, without the Lorentz boost), the supersymmetric localization formula of Lemma~\ref{lemm:Parisi-Sourlas} applies and yields
\begin{equation*}
    \left\langle \frac{z_a}{z_\delta}\mathbb{E}^{\widetilde{W},\widetilde{\bm{z}}}_{a}[g(\widetilde{\bm{L}})]\right\rangle_{\widetilde{W},v_{\delta}=0}=\mathbb{E}^{\widetilde{W},\bm{1}}_a[g(\widetilde{\bm{L}})].
\end{equation*}
Therefore, the corollary follows directly from Theorem~\ref{theo:Reinforced_Eisenbaum}.
\end{proof}

As a preliminary, record the following supersymmetric free field version of Eisenbaum's isomorphism theorem:
\begin{lemm}[Supersymmetric free field Eisenbaum's isomorphism theorem]
Consider a Markov jump process $Z$ starting at vertex $a$ on the augmented graph $\widetilde{V}$ with generator $\frac{1}{2}e^{-\widetilde{\bm{u}}}\widetilde{A^{\bm{u}}}e^{\widetilde{\bm{u}}}$, where $\widetilde{\bm{u}}$ is some fixed environment and $\widetilde{A^{\bm{u}}}$ is the matrix~\eqref{eq:Definition_AuMatrix_tilded}. Let $\widetilde{\bm{S}}$ denote the local times when the process is killed at the root vertex $\delta$. Then for any $s\in\mathbb{R}$ and any smooth bounded function $g$ with rapid decay,
\begin{equation*}
    \llbracket\mathbb{E}^{\widetilde{\bm{u}}}_{a}[g(\widetilde{\bm{S}}+\widetilde{\bm{z}}^2)]\rrbracket_{\widetilde{A^{\bm{u}}},X_{\delta}=\theta_s(0)}=\left\llbracket e^{u_\delta-u_a}\frac{x_a}{x_{\delta}}g(\widetilde{\bm{z}}^2)\right\rrbracket_{\widetilde{A^{\bm{u}}},X_{\delta}=\theta_s(0)},
\end{equation*}
where $\llbracket\cdot\rrbracket_{\widetilde{A^{\bm{u}}},X_{\delta}=\star}$ is the expectation with respect to the supersymmetric free field $(X_i)_{i\in \widetilde{V}}=((x_i,y_i,\xi_i,\eta_i))_{i\in \widetilde{V}}$ twisted by $A^{\widetilde{\bm{u}}}$ with pinning condition (defined at the beginning of Section~\ref{sec:a_supersymmetric_bayes_formula}) and $z_i^2=x_i^2+y_i^2+2\xi_i\eta_i+1$ for all $i\in \widetilde{V}$.
\end{lemm}
The proof is similar to that of Lemma~\ref{lemm:ClassicalSusy_BFSD} (by rewriting determinants in the original Eisenbaum's isomorphism theorem (see Theorem~\ref{theoA:standard_Eisenbaum}) arising from the Gaussian free field as fermionic integration) and is omitted.

\begin{proof}[Proof of Theorem~\ref{theo:Reinforced_Eisenbaum}]
The proof is very similar to that of Theorem~\ref{th:Reinforced_BFS-Dynkin}.

Start with the relation of local times $L_i=\sqrt{S_i+z_i^2}$,
\begin{equation*}
    \left\langle \frac{z_a}{z_\delta}\mathbb{E}^{\widetilde{W},\widetilde{\bm{z}}}_{a}[g(\widetilde{\bm{L}})]\right\rangle_{\widetilde{W},v_{\delta}=\theta_s(0)}=\left\langle \frac{z_a}{z_\delta}\mathbb{E}^{\widetilde{W},\widetilde{\bm{z}}}_{a}[g(\sqrt{\widetilde{\bm{S}}^2+\widetilde{\bm{z}}^2})]\right\rangle_{\widetilde{W},v_{\delta}=\theta_s(0)},
\end{equation*}
and the interpretation of the Vertex Reinforced Jump Process as a standard Markovian jump process in random environment,
\begin{equation*}
    \left\langle \frac{z_a}{z_\delta}\mathbb{E}^{\widetilde{W},\widetilde{\bm{z}}}_{a}[g(\sqrt{\widetilde{\bm{S}}^2+\widetilde{\bm{z}}^2})]\right\rangle_{\widetilde{W},v_{\delta}=\theta_s(0)}=\left\langle \frac{z_a}{z_\delta}\int_{\mathbb{R}^{|V|}}\mathbb{E}^{\bm{u}}_a[g(\sqrt{\widetilde{\bm{S}}^2+\widetilde{\bm{z}}^2})]d\nu^{\widetilde{W},\widetilde{\bm{z}}}_{a}(\widetilde{\bm{u}})\right\rangle_{\widetilde{W},v_{\delta}=\theta_s(0)}
\end{equation*}
where we can also rewrite as supersymmetric free field expectation in random environment using the supersymmetric Bayes formula of Theorem~\ref{th:susy_Bayes}:
\begin{equation*}
\begin{split}
    &\left\langle \frac{z_a}{z_\delta}\int_{\mathbb{R}^{|V|}}\mathbb{E}^{\widetilde{\bm{u}}}_{a}[g(\sqrt{\widetilde{\bm{S}}^2+\widetilde{\bm{z}}^2})]d\nu^{\widetilde{W},\widetilde{\bm{z}}}_{a}(\widetilde{\bm{u}})\right\rangle_{\widetilde{W},v_{\delta}=\theta_s(0)}\\
    ={}&\int_{\mathbb{R}^{|V|}}\left\llbracket\mathbb{E}^{\widetilde{\bm{u}}}_{a}[g(\sqrt{\widetilde{\bm{S}}^2+\widetilde{\bm{z}}^2})]\right\rrbracket_{\widetilde{A^{\bm{u}}},X_{\delta}=\theta_s(0)}d\nu^{\widetilde{W},\bm{1}}_{a}(\widetilde{\bm{u}}).
\end{split}
\end{equation*}

At this stage, use the supersymmetric free field version of Eisenbaum's isomorphism theorem,
\begin{equation*}
\begin{split}
    &\int_{\mathbb{R}^{|V|}}\left\llbracket\mathbb{E}^{\widetilde{\bm{u}}}_{a}[g(\sqrt{\widetilde{\bm{S}}^2+\widetilde{\bm{z}}^2})]\right\rrbracket_{\widetilde{A^{\bm{u}}},X_{\delta}=\theta_s(0)}d\nu^{\widetilde{W},\bm{1}}_{a}(\widetilde{\bm{u}})\\
    ={}&\int_{\mathbb{R}^{|V|}}\left\llbracket e^{u_\delta-u_a}\frac{x_a}{x_{\delta}}g(\widetilde{\bm{z}})\right\rrbracket_{\widetilde{A^{\bm{u}}},X_{\delta}=\theta_s(0)}d\nu^{\widetilde{W},\bm{1}}_{a}(\widetilde{\bm{u}}),
\end{split}
\end{equation*}
the change of base point formula~\eqref{eq:change_of_starting_point_mixing_measure},
\begin{equation*}
    \int_{\mathbb{R}^{|V|}}\left\llbracket e^{u_\delta-u_a}\frac{x_a}{x_{\delta}}g(\widetilde{\bm{z}})\right\rrbracket_{\widetilde{A^{\bm{u}}},X_{\delta}=\theta_s(0)}d\nu^{\widetilde{W},\bm{1}}_{a}(\widetilde{\bm{u}})=\int_{\mathbb{R}^{|V|}}\left\llbracket \frac{x_a}{x_{\delta}}g(\widetilde{\bm{z}})\right\rrbracket_{\widetilde{A^{\bm{u}}},X_{\delta}=\theta_s(0)}d\nu^{\widetilde{W},\bm{1}}_{\delta}(\widetilde{\bm{u}}),
\end{equation*}
and supersymmetric Bayes back to a $\text{H}^{2|2}$ expectation using Theorem~\ref{th:susy_Bayes},
\begin{equation*}
    \int_{\mathbb{R}^{|V|}}\left\llbracket \frac{x_a}{x_{\delta}}g(\widetilde{\bm{z}})\right\rrbracket_{\widetilde{A^{\bm{u}}},X_{\delta}=\theta_s(0)}d\nu^{\widetilde{W},\bm{1}}_{\delta}(\widetilde{\bm{u}})=\left\langle\frac{x_a}{x_{\delta}}g(\widetilde{\bm{z}})\int_{\mathbb{R}^{|V|}}d\nu^{\widetilde{W},\bm{z}}_{\delta}(\widetilde{\bm{u}})\right\rangle_{\widetilde{W},v_{\delta}=\theta_s(0)},
\end{equation*}
which integrates to $\langle\frac{x_a}{x_{\delta}}g(\widetilde{\bm{z}})\rangle_{\widetilde{W},v_{\delta}=\theta_s(0)}$ as $d\nu^{\widetilde{W},\widetilde{\bm{z}}}_{\delta}(\widetilde{\bm{u}})$ integrates to $1$ for each $\widetilde{\bm{z}}$.
\end{proof}

\section{Isomorphism theorems for the reinforced loop soup}\label{sec:LoopSoupIsomorphisms}
In this section, we study the reinforced loop soup constructed via the reinforced Wilson's algorithm defined in Section~\ref{sec:wilson_s_algorithm_with_reinforcement}. Recall that we denoted by $\mathbb{L}^{\bm{\vartheta}}_1$ the loop soup constructed as the erased loops during the reinforced Wilson's algorithm, where $\bm{\vartheta}$ denotes the initial local times at the start of the algorithm. We will always take $\bm{\vartheta}=\bm{1}$ in the following for simplicity, although our method extends naturally to arbitrary initial local times. Our goal is to establish a reinforced version of the standard Markovian loop soup isomorphism theorem. Our philosophy of proof is the same as in the previous section: we integrate the standard Markovian loop soup isomorphism theorem in a well-chosen random environment. The main novelty is that we correctly identify the reinforced Wilson's algorithm as the correct counterpart for the Vertex Reinforced Jump Process in the previous section, providing a trajectory explanation to this identity.

\subsection{Isomorphism theorems for the reinforced loop soup}\label{sub:isomorphism_theorems_for_the_reinforced_loop_soup}
By reinforced (or $\text{H}^{2|2}$) loop soup isomorphism theorem, we mean an identity relating the occupation field $\widehat{\mathbb{L}^{\bm{\vartheta}}_1}$ of the reinforced loop soup $\mathbb{L}^{\bm{\vartheta}}_1$ and certain observables of the $\text{H}^{2|2}$-model. The main result of this section is:
\begin{theo}[Supersymmetric hyperbolic Dynkin isomorphism for the reinforced loop soup]\label{th:SUSY_BFSD_Loop}
Denote by $\mathbb{E}^{\text{re-soup},\widetilde{W}}$ the expectation with respect to the law of the random reinforced loop soup $\mathbb{L}$ on the augmented graph $\widetilde{V}$ with edge weights $(\widetilde{W}_{ij})_{(ij)\in E(\widetilde{V})}$ and initial local times $\bm{1}$. Then for any smooth bounded function $g$ with rapid decay,
\begin{equation}
    \mathbb{E}^{\text{re-soup},\widetilde{W}}[g(\widehat{\mathbb{L}^{\bm{1}}_{1}})]=\left\langle g\left(\widetilde{\bm{x}}^2+\widetilde{\bm{y}}^2\right)\right\rangle_{\widetilde{W},\widetilde{v}_{\delta}=0},
\end{equation}
where the right hand side is the $\text{H}^{2|2}$-expectation form with pinning at $\delta\in\widetilde{V}$, and the notation $\widetilde{\bm{x}}^2+\widetilde{\bm{y}}^2$ should be understood component-wise, i.e. as $(x_i^2+y_i^2)_{i\in\widetilde{V}}$.
\end{theo}
In other words, the occupation field of the reinforced loop soup constructed out of the reinforced Wilson's algorithm is described by the variable $x^2+y^2$ of the $\text{H}^{2|2}$-model.

\subsection{Proof of the reinforced loop soup isomorphism theorem}
The proof follows roughly the same philosophy as those for the Vertex Reinforced Jump Process in Section~\ref{sec:isomorphism_theorems_for_the_vertex_reinforced_jump_process}. Consider the loop soup $\mathcal{L}^{\widetilde{\bm{u}}}_1$ on the augmented graph $\widetilde{V}$ with generator $\frac{1}{2}e^{-\widetilde{\bm{u}}}\widetilde{A^{\bm{u}}}e^{\widetilde{\bm{u}}}$ with $\widetilde{A^{\bm{u}}}$ defined in~\eqref{eq:Definition_AuMatrix_tilded} and $\widetilde{\bm{u}}$ is some field defined on the augmented graph $\widetilde{V}$ with the pinning condition $u_{\delta}=0$. Recall that the above generator corresponds to a standard Markov jump process on $\widetilde{V}$ with jump rate
\begin{equation*}
    \frac{1}{2}\widetilde{W}_{ij}e^{u_j-u_i}
\end{equation*}
and killed at the root vertex $\delta$. Denote by $\mathbb{E}^{\text{loop},\widetilde{\bm{u}}}$ the expectation with respect to the law of this Poissonian loop soup in the environment $\widetilde{\bm{u}}$.

Our first step is to upgrade the standard Markovian Dynkin isomorphism theorem for the loop soup (see Theorem~\ref{lemm:Classical_Loop_Dynkin}) to a supersymmetric free field version:
\begin{lemm}[Supersymmetric free field Dynkin isomorphism for the loop soup]\label{lemm:Upgrade_susy_loop}
For any smooth bounded function $g$ with rapid decay, if $X_i=(x_i,y_i,\xi_i,\eta_i)$ is the supersymmetric free field with the expectation form $\llbracket \cdot \rrbracket_{\widetilde{A^{\bm{u}}},\widetilde{X}_{\delta}=0}$ defined at the beginning of Section~\ref{sec:a_supersymmetric_bayes_formula}, we have (with $\alpha=1$)
\begin{equation*}
    \left\llbracket\mathbb{E}^{\text{loop},\widetilde{\bm{u}}}[g(\widehat{\mathcal{L}^{\bm{\widetilde{u}}}_{1}})]\right\rrbracket_{\widetilde{A^{\bm{u}}},\widetilde{X}_{\delta}=0}=\llbracket g(\widetilde{\bm{x}}^2+\widetilde{\bm{y}}^2)\rrbracket_{\widetilde{A^{\bm{u}}},\widetilde{X}_{\delta}=0}.
\end{equation*}
\end{lemm}

\begin{proof}[Proof of Lemma~\ref{lemm:Upgrade_susy_loop}]
The proof is similar to that of Lemma~\ref{lemm:ClassicalSusy_BFSD}: we include a sketch here.

Consider the symmetric generator $\widetilde{B^{u}}=e^{\widetilde{\bm{u}}}\widetilde{A^{\bm{u}}}e^{\widetilde{\bm{u}}}$ as in the beginning of Section~\ref{sec:a_supersymmetric_bayes_formula}. To this generator, the standard procedure of adding the determinant arising from fermionic integration explained in the beginning of Lemma~\ref{lemm:ClassicalSusy_BFSD} yields
\begin{equation*}
    \left\llbracket\mathbb{E}^{\text{loop},\widetilde{\bm{u}}}[g(\widehat{\mathcal{L}^{\bm{\widetilde{u}}}_{1}})]\right\rrbracket_{\widetilde{B^{\bm{u}}},\widetilde{X}_{\delta}=0}=\left\llbracket g\Big(\frac{1}{2}(\widetilde{\bm{x}}^2+\widetilde{\bm{y}}^2)\Big)\right\rrbracket_{\widetilde{B^{\bm{u}}},\widetilde{X}_{\delta}=0},
\end{equation*}
where we applied Theorem~\ref{lemm:Classical_Loop_Dynkin} to the generator $\widetilde{B^{u}}$. It remains to change the generator $\widetilde{B^{u}}$ back to $\widetilde{A^{u}}$. Recall $\llbracket F(X)\rrbracket_{\widetilde{A^{\bm{u}}}}=\llbracket F(e^{\widetilde{\bm{u}}}X)\rrbracket_{\widetilde{B^{\bm{u}}}}$ from~\eqref{eq:ShortHand_Au_Expectation} and the scaling of local times~\eqref{eq:Relation_AB_localtimes}. The above display becomes
\begin{equation*}
    \left\llbracket\mathbb{E}^{\text{loop},\widetilde{\bm{u}}}\left[g\Big(\frac{1}{2}e^{-2\widetilde{\bm{u}}}\widehat{\mathcal{L}^{\bm{\widetilde{u}}}_{1}}\Big)\right]\right\rrbracket_{\widetilde{A^{\bm{u}}},\widetilde{X}_{\delta}=0}=\left\llbracket g\Big(\frac{1}{2}e^{-2\widetilde{\bm{u}}}(\widetilde{\bm{x}}^2+\widetilde{\bm{y}}^2)\Big)\right\rrbracket_{\widetilde{A^{\bm{u}}},\widetilde{X}_{\delta}=0}.
\end{equation*}
Redefining the function $g$ by scaling yields Lemma~\ref{lemm:Upgrade_susy_loop}.
\end{proof}

It remains to use the random environment representation of the $\text{H}^{2|2}$-model~\cite{MR3420510} and integrate (i.e. anneal) the above isomorphism to get Theorem~\ref{th:SUSY_BFSD_Loop}.
\begin{proof}[Proof of Theorem~\ref{th:SUSY_BFSD_Loop}]
The proof combines a list of lemmas presented in this article.

$\bullet$ First, let us rewrite the reinforced loop soup occupation field as a mixture of standard Markovian loop soup occupation fields in some random environment:
\begin{equation*}
    \mathbb{E}^{\text{re-soup},\widetilde{W}}[g(\widetilde{\mathbb{L}^{\bm{1}}_{1}})]=\int_{\mathbb{R}^{|V|}}\mathbb{E}^{\text{soup},\widetilde{\bm{u}}}[g(\widehat{\mathcal{L}^{\bm{\widetilde{u}}}_{1}})]d\nu_{\delta}^{\widetilde{W},\bm{1}}(\widetilde{\bm{u}})
\end{equation*}
with $d\nu_{\delta}^{\widetilde{W},\bm{1}}$ the mixing measure defined in~\eqref{eq:killed_MixingDensityU_i0}. This identity follows from the interpretation of the reinforced Wilson's algorithm as a mixture of standard Markovian Wilson's algorithms in the same random environment, Lemma~\ref{lemm:mixture_of_Markov_Wilson}.

$\bullet$ Next, rewrite the integrand in the above display as
\begin{equation*}
    \mathbb{E}^{\text{soup},\widetilde{\bm{u}}}[g(\widehat{\mathcal{L}^{\bm{\widetilde{u}}}_{1}})]=\llbracket g\left((x_i^2+y_i^2)_{i\in\widetilde{V}}\right)\rrbracket_{\widetilde{A^{\bm{u}}},\widetilde{v}_{\delta}=0},
\end{equation*}
which follows from Lemma~\ref{lemm:Upgrade_susy_loop}.

$\bullet$ We continue with the supersymmetric Bayes formula of Theorem~\ref{th:susy_Bayes} (with $a=\delta$):
\begin{equation*}
    \int_{\mathbb{R}^{|V|}}\llbracket g\left((x_i^2+y_i^2)_{i\in\widetilde{V}}\right)\rrbracket_{\widetilde{A^{\bm{u}}},\widetilde{v}_{\delta}=0}d\nu_{\delta}^{\widetilde{W},\bm{1}}(\widetilde{\bm{u}})=\left\langle g\left((x_i^2+y_i^2)_{i\in\widetilde{V}}\right)\int_{\mathbb{R}^{|V|}}d\nu_{\delta}^{\widetilde{W},\widetilde{\bm{z}}}(\widetilde{\bm{u}})\right\rangle_{\widetilde{W},\widetilde{v}_{\delta}=0}.
\end{equation*}
Notice that the initial local time conditions for the mixing measure on the right-hand side above is the $\widetilde{\bm{z}}$ variable in the $\text{H}^{2|2}$-model.

$\bullet$ The proof finishes by combining the steps before to get
\begin{equation*}
\begin{split}
    \mathbb{E}^{\text{re-soup},\widetilde{W}}[g(\widetilde{\mathbb{L}^{\bm{1}}_{1}})]&=\left\langle g\left((x_i^2+y_i^2)_{i\in\widetilde{V}}\right)\int_{\mathbb{R}^{|V|}}d\nu_{\delta}^{\widetilde{W},\widetilde{\bm{z}}}(\widetilde{\bm{u}})\right\rangle_{\widetilde{W},\widetilde{v}_{\delta}=0}\\
    &=\left\langle g\left((x_i^2+y_i^2)_{i\in\widetilde{V}}\right)\right\rangle_{\widetilde{W},\widetilde{v}_{\delta}=0},
\end{split}
\end{equation*}
since $d\nu_{\delta}^{\widetilde{W},\widetilde{\bm{z}}}$ integrates to $1$.
\end{proof}

\subsection{Identities on the reinforced spanning tree}\label{sub:identities_on_the_reinforced_spanning_tree}
We also include an isomorphism theorem for the reinforced spanning tree, whose existence is constructed with the reinforced Wilson's algorithm defined in Section~\ref{sec:wilson_s_algorithm_with_reinforcement}. It can be seen as a counterpart of classical results (e.g. those in~\cite{MR2064357}) in the random environment corresponding to the mixing measure~\eqref{eq:MixingDensityU_i0}. We will only illustrate this with the most celebrated Kirchoff's law, adapted to the $H^{2|2}$ model.

\begin{theo}[Kirchoff's theorem for the reinforced spanning tree]
Denote by $\mathbb{E}^{\text{re-tree},\widetilde{W}}$ the expectation with respect to the law of the reinforced spanning tree $\mathbb{T}^{\bm{1}}$ on the augmented graph $\widetilde{V}$ with edge weights $(\widetilde{W}_{ij})_{(ij)\in E(\widetilde{V})}$ and initial local times $\bm{1}$. Then for any collection of edges $S\subset E(\widetilde{V})$, 
\begin{equation}
  \mathbb{P}^{\text{re-tree},\widetilde{W}}[S\subset\mathbb{T}^{\bm{1}}]=\left\langle \prod_{(ij)\in S}\widetilde{W}_{ij}e^{u_{i}+u_{j}}(\overline{\psi}_{i}-\overline{\psi}_{j})(\psi_{i}-\psi_{j})\right\rangle_{\widetilde{W},\widetilde{v}_{\delta}=0},
\end{equation}
where the right hand side is the $\text{H}^{2|2}$-expectation form with pinning at $\delta\in\widetilde{V}$, expressed in the horospherical coordinates in Theorem \ref{lemm:Link_VRJP_H22}.
\end{theo}

\begin{proof}
Recall that $\nu_{\delta}^{\widetilde{W},\bm{1}}(\bm{u})$ defined in~\eqref{eq:killed_MixingDensityU_i0} is the marginal law of \(\left< \cdot \right>_{\widetilde{W},\widetilde{v}_{\delta}=0} \) in the horospherical coordinates. When completed with the other horospherical variables \(s,\overline{\psi},\psi \), we have
\[\begin{aligned}
&\left< \prod_{(ij)\in S}\widetilde{W}_{ij}e^{u_{i}+u_{j}} (\overline{\psi}_{i}-\overline{\psi}_{j})(\psi_{i}-\psi_{j})\right>_{\widetilde{W},\widetilde{v}_{\delta}=0}\\
={}&\int \frac{1}{D(\widetilde{W},\widetilde{\bm{u}})} \prod_{(ij)\in S} \widetilde{W}_{ij}e^{u_{i}+u_{j}} (\overline{\psi}_{i}-\overline{\psi}_{j})(\psi_{i}-\psi_{j}) \cdot D(\widetilde{W},\widetilde{\bm{u}}) \mathbf{1}_{\{u_{\delta}=s_{\delta}=\psi_{\delta}=\overline{\psi}_{\delta}=0\}}\\
& \times e^{-\frac{1}{2}\sum_{(ij)\in E(\widetilde{V})} \widetilde{W}_{ij}(e^{u_{i}-u_{j}}+e^{u_{j}-u_{i}}-2+e^{u_{i}+u_{j}}[(s_{i}-s_{j})^{2}+2(\overline{\psi}_{i}-\overline{\psi}_{j})(\psi_{i}-\psi_{j})])} \prod_{i\in V} d u_i d s_i d \overline{\psi}_{i} d\psi_{i}  \\
={}&\int \frac{1}{D(\widetilde{W},\widetilde{\bm{u}})}  \prod_{(ij)\in S} \widetilde{W}_{ij}e^{u_{i}+u_{j}}  \\
& \times  \int \prod_{(ij)\in S} (\overline{\psi}_{i}-\overline{\psi}_{j})(\psi_{i}-\psi_{j}) \mathbf{1}_{\{\psi_{\delta}=\overline{\psi}_{\delta}=0\}}e^{-\sum_{(ij)\in E(\widetilde{V})} \widetilde{W}_{ij}e^{u_{i}+u_{j}} (\overline{\psi}_{i}-\overline{\psi}_{j})(\psi_{i}-\psi_{j})}\prod_{i\in V}d\overline{\psi}_i d\psi_i \\
& \times  D(\widetilde{W}, \widetilde{\bm{u}}) \cdot \mathbf{1}_{\{u_{\delta}=s_{\delta}=0\}}e^{-\frac{1}{2}\sum_{(ij)\in E(\widetilde{V})} \widetilde{W}_{ij}(e^{u_{i}-u_{j}}+e^{u_{j}-u_{i}}-2+e^{u_{i}+u_{j}}[(s_{i}-s_{j})^{2}])} \prod_{i\in V} d u_i d s_i.
  \end{aligned}\]
The supersymmetric free field version of isomorphism theorem for the random spanning tree is well-known: 
\[\begin{aligned}
 & \int \prod_{(ij)\in S} (\overline{\psi}_{i}-\overline{\psi}_{j})(\psi_{i}-\psi_{j})\mathbf{1}_{\{\psi_{\delta}=\overline{\psi}_{\delta}=0\}} e^{-\sum_{(ij)\in E(\widetilde{V})} \widetilde{W}_{ij}e^{u_{i}+u_{j}} (\overline{\psi}_{i}-\overline{\psi}_{j})(\psi_{i}-\psi_{j}) } \prod_{i\in V}d\overline{\psi}_i d\psi_i\\
={}& \sum_{B: B\cup S= \widetilde{T}} \prod_{(ij)\in B} W_{ij}e^{u_{i}+u_{j}},
  \end{aligned}\]
where we sum over \(B\subset E\) such that \(B\cup S \) becomes a spanning tree $\widetilde{T}$ of $\widetilde{V}$. The proof of the above identity can be found in e.g.~\cite[Corollary~B.3]{MR4218682}.

By Gaussian integral,
\begin{small}
\[ D(\widetilde{W}, \widetilde{\bm{u}})\int \mathbf{1}_{\{u_{\delta}=s_{\delta}=0\}}e^{-\frac{1}{2}\sum_{(ij)\in E(\widetilde{V})} \widetilde{W}_{ij}(e^{u_{i}-u_{j}}+e^{u_{j}-u_{i}}-2+e^{u_{i}+u_{j}}[(s_{i}-s_{j})^{2}])} \prod_{i\in V} d s_i  du_i = d\nu_{\delta}^{\widetilde{W},\bm{1}}( \bm{u}).\]
\end{small}

Therefore, 
\[\begin{aligned}
  &\left< \prod_{(ij)\in S}\widetilde{W}_{ij}e^{u_{i}+u_{j}} (\overline{\psi}_{i}-\overline{\psi}_{j})(\psi_{i}-\psi_{j})\right>_{\widetilde{W},\widetilde{v}_{\delta}=0}\\
={}&\int \frac{1}{D(\widetilde{W},\widetilde{\bm{u}})} \prod_{(ij)\in S} \widetilde{W}_{ij}e^{u_{i}+u_{j}} \sum_{B: B\cup S= \widetilde{T}} \prod_{(ij)\in B} W_{ij}e^{u_{i}+u_{j}}  d\nu^{\widetilde{W},\bm{1}}_{\delta}(\bm{u}) \\
={}&\int   \mathbb{P}^{\widetilde{\bm{u}}}(S\subset \widetilde{T}) d\nu^{\widetilde{W},\bm{1}}_{\delta}(\bm{u}),
  \end{aligned}\]
where \( \mathbb{P}^{\widetilde{\bm{u}}}\) is the law of the random spanning tree \(\widetilde{T} \) on $\widetilde{V}$ with edge weights \(\widetilde{W}_{ij}e^{u_i+u_j} \). By construction, the law of the reinforced spanning tree is a mixture of random spanning trees with edge weights $\widetilde{W}_{i,j}e^{u_{i}+u_{j}}$ where the collection of $\{u_{i}\}_{i\in \widetilde{V}}$ is sampled by $\nu_{\delta}^{\widetilde{W},\bm{1}} (\bm{u})$, therefore the last display is $\mathbb{P}^{\text{re-tree},\widetilde{W}}[S\subset\mathbb{T}^{\bm{1}}]$.
\end{proof}

\begin{rema}
We make a brief discussion on the expressions in the isomorphism theorems that appeared in this paper.

In Section~\ref{sub:isomorphism_theorems_for_the_reinforced_loop_soup}, we see that the reinforced loop soup is related to the ``bosonic part'' of the $H^{2|2}$ spins, especially its local time is related to the observable $x^{2}+y^{2}$  in the $H^{2|2}$ model. In contrast, Section~\ref{sub:identities_on_the_reinforced_spanning_tree} shows that the configuration of the reinforced spanning tree is related to the ``fermionic part'' of the $H^{2|2}$ spins decorated by the \(u \)-variables, in that its statistics can be derived using functionals involving only the $\xi,\eta,u$ components of the $H^{2|2}$ spins (recall that $\xi_{i}=\psi_{i}e^{u_{i}}$ and $\eta_{i}=\overline{\psi}_{i}e^{u_{i}}$). The $H^{2|2}$ model involves non-trivial interactions between the bosonic and fermonic variables, and the interaction between the reinforced loop soup and the reinforced spanning tree  is complicated (see~\cite[Section~9.8]{MR4789605} for other examples).

One can also ponder on the role of the $z$ variable, which appears in the three random walk $H^{2|2}$ isomorphism theorems in Section~\ref{sec:isomorphism_theorems_for_the_vertex_reinforced_jump_process}. In the standard Markovian case, these isomorphism theorems can be derived using the standard Markovian loop soup isomorphism (see~\cite[Section~6.2]{MR4789605} and especially Remark~6.6 therein).

The supersymmetric free field versions of these isomorphism theorems, provided in Section~\ref{sec:isomorphism_theorems_for_the_vertex_reinforced_jump_process}, are obtained from these theorems by performing a translation by $2\xi\eta+1$, which maps $x^2+y^2$ to $z^2$. Coming to the $H^{2|2}$ model, the $z$ variable plays the role of the initial local time in the Vertex Reinforced Jump Process, as $L_i(s)=\sqrt{S_i(t)+z_i^{2}}$, see e.g.~\eqref{eq:changed_local_times}. This allows one to perform the time change and obtain the three isomorphism theorems in Section~\ref{sec:isomorphism_theorems_for_the_vertex_reinforced_jump_process} in the original time-scale $\mathbf{L}$, and the local times are related to observable $z$ of the $H^{2|2}$ model.

Notice that it is not enough to have only the local times of the reinforced loop soup (i.e. represented by the variables $x^{2}+y^{2}$) to perform the time change: information about the tree (in particular, the variables $\xi,\eta$) is missing. This is the reason why our reinforced Wilson's algorithm defined in Section~\ref{sec:wilson_s_algorithm_with_reinforcement} needs to be generated by a single long trajectory of the Reinforced Vertex Jump Process.\footnote{A tedious computation, not reproduced in this article, shows that if we define the reinforced Wilson's algorithm naïvely using the same recipe as in the standard Markovian case is not partially exchangeable, which means that no suitable time-change can be constructed}
\end{rema}

\section{Some extensions of our results}\label{sec:some_extensions_of_our_results}
Our methodology is quite general and our results with the reinforced Wilson's algorithm generalize naturally in several directions. We list several extensions with the reinforced loop soup, especially we explain below the reconstruction of the reinforced loop soup (the realization of the random loop configuration, not just the occupation time field thereof) from the reinforced Wilson's algorithm. Other possible directions involving e.g. vertex diminished jump process~\cite{MR3520013,MR4255180} and reinforced random interlacements (roughly speaking, reinforced loop soup going through infinity, see~\cite[Section~4.5]{MR2932978}): these can be studied using similar steps and we do not discuss them further in this note.

\subsection{Construction of general reinforced loop soup $\mathbb{L}_{\alpha}$ via the reinforced Wilson's algorithm}\label{subse:construction_of_general_reinforced_loop_soup_texorpdfstring_mathcal_l}
Following an idea of Le Jan of constructing the standard Poissonian loop soup with arbitrary parameter $\alpha>0$ using the standard Markovian Wilson's algorithm, we can define the reinforced loop soup $\mathbb{L}_{\alpha}$ (therefore also the reinforced loop soup occupation time field $\widehat{\mathbb{L}_{\alpha}}$) with any $\alpha>0$ using the reinforced Wilson's algorithm introduced in Section~\ref{sec:wilson_s_algorithm_with_reinforcement}. We postpone the reconstruction of the loop soup until this point because Theorem~\ref{th:SUSY_BFSD_Loop} only concerns the occupation time field, but from this point on we need to be really careful about the erased loops in Wilson's algorithm.

\medskip

$\bullet$ First, we need to be more specific about the loops erased in the reinforced Wilson's algorithm defined in Section~\ref{subse:WilsonWithOneExploration}. Since our Theorem~\ref{th:SUSY_BFSD_Loop} only concerns the occupation time field of the reinforced loop soup $\mathbb{L}_1$, we had the freedom of concatenating loops into a large loop, or decomposing a large loop into smaller loops (including trivial loops that contain no jump). We now follow~\cite[Section~8.3]{MR4789605} (see also~
\cite[Remark~18]{MR2815763}) to explain the reconstruction of the reinforced loop soup $\mathbb{L}_1$ using the decomposition of loops with Poisson-Dirichlet$(0,1)$ distribution.

The reconstruction starts with running the reinforced Wilson's algorithm until the set of erased loops $\mathbb{O}$ is generated using the algorithm in Section~\ref{subse:WilsonWithOneExploration}: this is done almost surely in finite time by the recurrence of the Vertex Reinforced Jump Process. Order the vertices as $i_1,i_2,\dots,i_{|V|}$ in the following way. Start with $E_0=\{\delta\}$ and let $i_1$ be the first vertex that the jump process visits outside of $E_0$ at time $\tau_1^{-}$, and refine $E_0$ as $E_1=\{\delta,i_1\}$. Since the jump process finishes in finite time, there is a last visit time to $i_1$ that we denote by $\tau_1^{+}$. Let $i_2$ be the first vertex that the jump process visits outside of $E_1$ after $\tau_1$ at time $\tau_2^{-}$, define $E_2=\{\delta,i_1,i_2\}$ and denote the last visit time to $i_2$ as $\tau_2^{+}$. Continue this process until the spanning tree is generated: at this point all the vertices are ordered since the loop erasure procedure at time $\tau_k^{+}$ only leaves a tree with vertices $E_k$, so Wilson's algorithm cannot stop when $k\neq |V|$.

After this ordering, regroup the erased loops $\mathbb{O}$ in the following manner: let $\mathbb{O}_{i_1}$ be one large loop formed by the exploration of the jump process between time $\tau_1^{-}$ and $\tau_1^{+}$, then erase (all components forming) $\mathbb{O}_{i_1}$ from $\mathbb{O}$ to obtain a new collection of loops $\mathbb{O}^{{i_1}^{c}}$; next, let $\mathbb{O}_{i_2}$ be one large loop formed by the exploration of the jump process between time $\tau_2^{-}$ and $\tau_2^{+}$ in $\mathbb{O}^{{\{i_1\}}^{c}}$, and erase (all components forming) $\mathbb{O}_{i_2}$ from $\mathbb{O}^{{i_1}^{c}}$ to obtain a new collection of loops $\mathbb{O}^{{\{i_1,i_2\}}^{c}}$; continue until we obtain $\mathbb{O}_{i_{|V|}}$ and $\mathbb{O}^{{V}^{c}}$. Now $\mathbb{O}^{{V}^{c}}$ is empty since any loop in $\mathbb{O}$ must intersect at least once with some vertex in $V$. Therefore, $\mathbb{O}_{i_{1}}, \mathbb{O}_{i_{2}}, \dots, \mathbb{O}_{i_{|V|}}$ form a collection of $|V|$ large loops that has the same occupation time field as $\mathbb{O}$. We then perform the Poisson-Dirichlet$(0,1)$ decomposition of all these large loops $(\mathbb{O}_{i_k})_{k=1,\dots,|V|}$ individually independently as in~\cite[Theorem~8.3]{MR4789605}.
\begin{lemm}\label{lemm:Reinforced_PoissonDirichlet}
The collection of all loops obtained in this manner is exactly $\mathbb{L}_1$, in the sense that it is equal in law to the annealed version of the standard Poissonian loop soup $\mathcal{L}_1$ in the correct random environment given by the mixing measure~\eqref{eq:killed_MixingDensityU_i0}.
\end{lemm}
\begin{proof}
When the Vertex Reinforced Jump Process in the above description is replaced by the standard Markovian jump process, we obtain exactly $\mathcal{L}_1$ by~\cite[Theorem~8.3]{MR4789605}. Now since the law of the Poisson-Dirichlet$(0,1)$ is independent of the random mixing measure~\eqref{eq:killed_MixingDensityU_i0}, integrating $\mathcal{L}_1$ with respect to this random environment as in Theorem~\ref{th:SUSY_BFSD_Loop} yields $\mathbb{L}_1$.
\end{proof}

\medskip

$\bullet$ Suppose now that $\alpha=k\in\mathbb{Z}_{>0}$ is a positive integer. We now explain how to obtain the reinforced loop soup $\mathcal{L}_2$ with $k=2$ with the reinforced Wilson's algorithm in Section~\ref{sec:wilson_s_algorithm_with_reinforcement}, and the general case is similar. Recall the reinforced Wilson's algorithm with one single exploration of some Vertex Reinforced Jump Process from Section~\ref{subse:WilsonWithOneExploration}, and modify it in the following way: when the spanning tree $\mathbb{T}$ (and the dual random collection of erased loops $\mathbb{O}$) is formed from this algorithm, continue the same Vertex Reinforced Jump Process until the next time it goes back to the root vertex $\delta$, and start the loop erasure procedure afresh until it generates another spanning tree $\mathbb{T}'$ (and another random collection of erased loops $\mathbb{O}'$ dual to $\mathbb{T}'$). Notice that $\mathbb{O}$ and $\mathbb{O}'$ are not independent, but for each of them, we can first regroup the loops as $\mathbb{O}_{i_{1}}, \mathbb{O}_{i_{2}}, \dots, \mathbb{O}_{i_{|V|}}$ and $\mathbb{O}'_{j_{1}}, \mathbb{O}'_{j_{2}}, \dots, \mathbb{O}'_{j_{|V|}}$ like in Lemma~\ref{lemm:Reinforced_PoissonDirichlet} (where the ordering of the vertices might differ), perform Poisson-Dirichlet$(0,1)$ decompositions independently on each of these $2|V|$ large loops and obtain a loop soup. We claim that:
\begin{lemm}\label{lemm:Reinforce_loopsoup_integer}
The collection of all loops obtained in this manner is exactly $\mathbb{L}_2$, in the sense that it is equal in law to the annealed version of the standard Poissonian loop soup $\mathcal{L}_2$ in the correct random environment given by the mixing measure~\eqref{eq:killed_MixingDensityU_i0}.
\end{lemm}
\begin{proof}
When the Vertex Reinforced Jump Process in the above description is replaced by the standard Markovian jump process, we obtain exactly $\mathcal{L}_2$ by the Poisson point process definition of the Poissonian loop soup, Section~\ref{subse:poissonian_loop_soup}. Now since the reinforced Wilson's algorithm is generated by a single long Vertex Reinforced Jump Process, the mixing measure of the random environment is~\eqref{eq:killed_MixingDensityU_i0}. Furthermore, the concatenation process between the two parts of the reinforced Wilson's algorithm as well as the Poisson-Dirichlet$(0,1)$ law are independent of the mixing measure~\eqref{eq:killed_MixingDensityU_i0}, therefore we obtain $\mathbb{L}_{2}$ via the description above by integrating in the random environment similarly to Theorem~\ref{th:SUSY_BFSD_Loop}.
\end{proof}
In general, for $k>2$, it suffices to continue the reinforced Wilson's algorithm $k$ times until it forms $k$ spanning trees, and repeat the procedure similarly as above to obtain the reinforced loop soup $\mathbb{L}_k$.

\medskip

$\bullet$ Finally, suppose that $\alpha\notin\mathbb{Z}_{>0}$ and let $k\in\mathbb{Z}_{>0}$ be such that $k-1<\alpha<k$. It is known for the standard Poissonian loop soup that one can decrease the parameter $\alpha$ using the thinning procedure of the Poisson point process. We explain the case where $k=1$ (thus $\alpha\in(0,1)$) since the general case is similar. Consider the reinforced loop soup $\mathbb{L}_1$ obtained in Lemma~\ref{lemm:Reinforced_PoissonDirichlet}. Construct a new loop soup $\mathbb{L}'_{\alpha}$ from $\mathbb{L}_1$ using the following procedure: for each loop in $\mathbb{L}_1$, it has probability $\alpha$ of being in $\mathbb{L}'_{\alpha}$ and probability $1-\alpha$ of being discarded, independently of all other loops in $\mathbb{L}_1$.
\begin{lemm}
The law of $\mathbb{L}'_{\alpha}$ is exactly the reinforced loop soup $\mathbb{L}_{\alpha}$, in the sense that it is the annealed version of standard Poissonian loop soup $\mathcal{L}_{\alpha}$ in the correct mixing environment~\eqref{eq:killed_MixingDensityU_i0}.
\end{lemm}
\begin{proof}
It is known that $\mathcal{L}_{\alpha}$ can be obtained from $\mathcal{L}_{1}$ when $\alpha\in(0,1)$ from a thinning procedure exactly as described above: this directly comes from the definition of the Poissonian loop soup $\mathcal{L}_{\alpha}$ as a Poisson point process with intensity $\alpha \mu^{\circ}$, see Section~\ref{subse:poissonian_loop_soup}. Since the thinning procedure is independent of the mixing environment~\eqref{eq:killed_MixingDensityU_i0}, integrating yields the reinforced loop soup $\mathbb{L}_{\alpha}$.
\end{proof}
In general, if $\alpha\in(k-1,k)$, one can first construct the reinforced loop soup $\mathbb{L}_k$ as in Lemma~\ref{lemm:Reinforce_loopsoup_integer} then perform a similar thinning procedure to get the reinforced loop soup $\mathbb{L}_{\alpha}$ with non-integer parameter $\alpha$ from $\mathbb{L}_k$.

\subsection{Generalized reinforced loop soup isomorphism theorems for integer-valued parameters}
We can generalize the reinforced loop soup isomorphism for the $\text{H}^{2|2}$-model to general $\text{H}^{2k|2k}$ models for any integer $k\geq 1$, using an extension of the reinforced Wilson's algorithm described in Section~\ref{subse:construction_of_general_reinforced_loop_soup_texorpdfstring_mathcal_l}.

First, let us recall briefly the definition of the $\text{H}^{n|2m}$ model with $n+m>0$ following the notations of~\cite[Section~2.4]{MR4218682}. It is defined similarly to the $\text{H}^{2|2}$-model, but now we deal with a $(n+2m)$-component supervector
\begin{equation*}
    \bm{v}=(\phi_1,\phi_2,\dots,\phi_n,\xi_1,\eta_1,\dots,\xi_m,\eta_m)
\end{equation*}
with the $\phi_i$ coordinates bosonic variables and $\xi_i,\eta_i$ coordinates fermionic variables. The notation $\bm{v}$ will denote a superspin in $\text{H}^{n|2m}$ instead of $\text{H}^{2|2}$ in this subsection only. Introduce the distinguished vector in the positive branch as before,
\begin{equation}\label{eq:Distinguished_z}
    \forall 1\leq i\leq n, \quad z_i=\left(\sum_{j=1}^{n}\phi_j^2+2\sum_{k=1}^{m}\xi_k\eta_k+1\right)^{\frac{1}{2}}.
\end{equation}
The inner product on $\text{H}^{n|2m}$ is defined as\footnote{Notice that this definition forces the number of fermionic variables to be even, although standard Poissonian loop soup isomorphism theorems exist for half-integers $\alpha$ (and not only integers), we cannot generalize them directly to the supersymmetric hyperbolic sigma models with our setup.}
\begin{equation*}
    \bm{v}\cdot\bm{v}'=\sum_{j=1}^{n}\phi_{j}\phi'_{j}+\sum_{k=1}^{m}(\xi_k\eta'_k+\xi'_k\eta_k)-\sum_{l=1}^{n}z_{l}z'_{l},
\end{equation*}
and notice that this implies the hyperbolic constraint $-\bm{v}^2=\bm{1}$. The energy or action functional of the $\text{H}^{n|2m}$ model is given by
\begin{equation*}
    \frac{1}{2}\bm{v}\Delta_W\bm{v}=\sum_{i,j\in V}v_i (\Delta_{W})_{ij} v_j
\end{equation*}
with the Berezin integral form (where $\phi_{i,j}$, $\xi_{i,j}$ and $\eta_{i,j}$ are the $j$-th components of the supervectors $\phi_i, \xi_i$ and $\eta_i$ with $i\in V$)
\begin{equation*}
    D\mu(\bm{v})=\prod_{i\in V}\left(\frac{1}{2\pi z_i}\prod_{j=1}^{n}d\phi_{i,j}\prod_{k=1}^{m}d\xi_{i,k}d\eta_{i,k}\right).
\end{equation*}
The (super-)expectations with pinning condition are defined analogously by attaching an extra vertex $\delta\in\widetilde{V}$ and imposing that $v_{\delta}=0$.

We have the following isomorphism theorem for the reinforced loop soup $\mathbb{L}_{k}$ with integer-valued parameters $k\in\mathbb{Z}_{>0}$:
\begin{theo}[Generalized reinforced loop soup isomorphism theorem]
Let $\widehat{\mathbb{L}_k}$ be the occupation time field of the reinforced loop soup $\mathbb{L}_{k}$ with initial local times $\bm{1}$. Then for any smooth bounded function $g$ with rapid decay,
\begin{equation*}
    \mathbb{E}^{\text{re-soup},\widetilde{W}}[g(\widehat{\mathbb{L}_{k}})]=\left\langle g\Big((\sum_{j=1}^{2k}\phi_{i,j}^2)_{i\in V}\Big)\right\rangle_{\widetilde{W},\widetilde{v}_{\delta}=0}
\end{equation*}
where $\langle\cdot\rangle_{\widetilde{W},\widetilde{v}_{\delta}=0}$ is the (super-)expectation with pinning at $\delta$ of the $\text{H}^{2k|2k}$ model (and $\phi_{i,j}$ is the $j$-th component of the supervector $\phi_i$ with $i\in V$). 
\end{theo}
\begin{proof}
We only indicate the main steps of the proof and leave the details to the reader.

The core of the proof is to generalize the supersymmetric Bayes formula, Theorem~\ref{th:susy_Bayes}, to the $\text{H}^{2k|2k}$ model. This requires introducing the corresponding $(2k,2k)$-supersymmetric free field, which is defined analogously to the $(2,2)$-supersymmetric free field of Section~\ref{subse:supersymmetric_free_field}. Then one should check that the cancellations in the proof of Theorem~\ref{th:susy_Bayes} generalize here, which is an elementary computation. The upshot is
\begin{equation*}
    \left\langle\int \,g\,d\nu^{\widetilde{W},\widetilde{\bm{z}}}_{\delta}(\widetilde{\bm{u}})\right\rangle_{\widetilde{W},v_{\delta}=0}=\int \llbracket\,g\,\rrbracket_{\widetilde{A^{\bm{u}}},X_{\delta}=0} d\nu^{\widetilde{W},\bm{1}}_{\delta}(\widetilde{\bm{u}}),
\end{equation*}
for any smooth function function $g$ with rapid decay, and now the brackets $\llbracket \cdot\rrbracket$ and $\langle\cdot\rangle$ denote respectively $(2k,2k)$-supersymmetric free field expectation with respect to $\widetilde{\bm{X}}$ and $\text{H}^{2k|2k}$ expectation with respect to $\widetilde{\bm{v}}$ (with pinning conditions in the subscript). The definition of $\llbracket \cdot\rrbracket_{\widetilde{A^{\bm{u}}},X_{\delta}=0}$ is similar to the one in the beginning of Section~\ref{sec:a_supersymmetric_bayes_formula}, and $d\nu^{\widetilde{W},\widetilde{\bm{z}}}_{\delta}(\widetilde{\bm{u}})$ is still the mixing measure in~\eqref{eq:killed_MixingDensityU_i0} only that $\widetilde{\bm{z}}$ is now the distinguished $\text{H}^{2k|2k}$ vector of~\eqref{eq:Distinguished_z}.

Then, we should generalize the isomorphism theorems for the standard Poissonian loop soup for $\mathcal{L}_{k}$~\cite[Remark~6.6(e)]{MR4789605} to its supersymmetric free field counterpart. This involves only rewriting some determinants using the fermionic variables, and is done similarly to the proof of Lemma~\ref{lemm:ClassicalSusy_BFSD}. The result is
\begin{equation*}
    \mathbb{E}^{\text{soup},\widetilde{\bm{u}}}[g(\widehat{\mathcal{L}^{\bm{\widetilde{u}}}_{k}})]=\left\llbracket g\Big((\sum_{j=1}^{2k}\phi_{i,j}^2)_{i\in V}\Big)\right\rrbracket_{\widetilde{A^{\bm{u}}},\widetilde{v}_{\delta}=0}.
\end{equation*}

Furthermore, we should use the interpretation of the reinforced Wilson's algorithm as a mixture of standard Markovian Wilson's algorithms in some random environment. Since the generalized reinforced Wilson's algorithm in Section~\ref{subse:construction_of_general_reinforced_loop_soup_texorpdfstring_mathcal_l} used to generate the reinforced loop soup $\mathbb{L}_{k}$ comes from one single Vertex Reinforced Jump Process starting from the root vertex $\delta$, the mixing measure is the same as~\eqref{eq:killed_MixingDensityU_i0}. Therefore, putting all the pieces together, the upshot is
\begin{equation*}
\begin{split}
    \mathbb{E}^{\text{re-soup},\widetilde{W}}[g(\widehat{\mathbb{L}_{k}})]&=\int \mathbb{E}^{\text{soup},\widetilde{\bm{u}}}[g(\widehat{\mathcal{L}^{\bm{\widetilde{u}}}_{k}})]d\nu_{\delta}^{\widetilde{W},\bm{1}}(\widetilde{\bm{u}})\\
    &=\int \left\llbracket g\Big((\sum_{j=1}^{2k}\phi_{i,j}^2)_{i\in V}\Big)\right\rrbracket_{\widetilde{A^{\bm{u}}},\widetilde{v}_{\delta}=0}d\nu_{\delta}^{\widetilde{W},\bm{1}}(\widetilde{\bm{u}})\\
    &=\left\langle g\Big((\sum_{j=1}^{2k}\phi_{i,j}^2)_{i\in V}\Big)\int d\nu_{\delta}^{\widetilde{W},\widetilde{\bm{z}}}(\widetilde{\bm{u}})\right\rangle_{\widetilde{W},\widetilde{v}_{\delta}=0}
\end{split}
\end{equation*}
and the proof is complete since $d\nu_{\delta}^{\widetilde{W},\widetilde{\bm{z}}}$ integrates to $1$.
\end{proof}

\subsection*{Conflict of interest}
The authors declare no conflicts of interest.

\bibliographystyle{alpha}

\begin{thebibliography}{EKM{\etalchar{+}}00}

\bibitem[Abd04]{MR2064357}
Abdelmalek Abdesselam.
\newblock The {G}rassmann-{B}erezin calculus and theorems of the matrix-tree
  type.
\newblock {\em Adv. in Appl. Math.}, 33(1):51--70, 2004.

\bibitem[ACK14]{MR3189433}
Omer Angel, Nicholas Crawford, and Gady Kozma.
\newblock Localization for linearly edge reinforced random walks.
\newblock {\em Duke Math. J.}, 163(5):889--921, 2014.

\bibitem[Aiz82]{MR678000}
Michael Aizenman.
\newblock Geometric analysis of {$\varphi ^{4}$} fields and {I}sing models.
  {I}, {II}.
\newblock {\em Comm. Math. Phys.}, 86(1):1--48, 1982.

\bibitem[BCHS21]{MR4218682}
Roland Bauerschmidt, Nicholas Crawford, Tyler Helmuth, and Andrew Swan.
\newblock Random spanning forests and hyperbolic symmetry.
\newblock {\em Comm. Math. Phys.}, 381(3):1223--1261, 2021.

\bibitem[BFS82]{MR648362}
David Brydges, J\"{u}rg Fr\"{o}hlich, and Thomas Spencer.
\newblock The random walk representation of classical spin systems and
  correlation inequalities.
\newblock {\em Comm. Math. Phys.}, 83(1):123--150, 1982.

\bibitem[BFS83]{MR719815}
David~C. Brydges, J\"{u}rg Fr\"{o}hlich, and Alan~D. Sokal.
\newblock The random-walk representation of classical spin systems and
  correlation inequalities. {II}. {T}he skeleton inequalities.
\newblock {\em Comm. Math. Phys.}, 91(1):117--139, 1983.

\bibitem[BHS19]{MR4021254}
Roland Bauerschmidt, Tyler Helmuth, and Andrew Swan.
\newblock Dynkin isomorphism and {M}ermin-{W}agner theorems for hyperbolic
  sigma models and recurrence of the two-dimensional vertex-reinforced jump
  process.
\newblock {\em Ann. Probab.}, 47(5):3375--3396, 2019.

\bibitem[BHS21]{MR4255180}
Roland Bauerschmidt, Tyler Helmuth, and Andrew Swan.
\newblock The geometry of random walk isomorphism theorems.
\newblock {\em Ann. Inst. Henri Poincar\'{e} Probab. Stat.}, 57(1):408--454,
  2021.

\bibitem[CD87]{coppersmith1987random}
Don Coppersmith and Persi Diaconis.
\newblock {Random walk with reinforcement}.
\newblock {\em Unpublished manuscript}, pages 187--220, 1987.

\bibitem[CJS{\etalchar{+}}04]{MR2110547}
Sergio Caracciolo, Jesper~Lykke Jacobsen, Hubert Saleur, Alan~D. Sokal, and
  Andrea Sportiello.
\newblock Fermionic field theory for trees and forests.
\newblock {\em Phys. Rev. Lett.}, 93(8):080601, 4, 2004.

\bibitem[Dav90]{MR1030727}
Burgess Davis.
\newblock Reinforced random walk.
\newblock {\em Probab. Theory Related Fields}, 84(2):203--229, 1990.

\bibitem[DF80]{MR556418}
P.~Diaconis and D.~Freedman.
\newblock de {F}inetti's theorem for {M}arkov chains.
\newblock {\em Ann. Probab.}, 8(1):115--130, 1980.

\bibitem[DF84]{MR786142}
P.~Diaconis and D.~Freedman.
\newblock Partial exchangeability and sufficiency.
\newblock In {\em Statistics: applications and new directions ({C}alcutta,
  1981)}, pages 205--236. Indian Statist. Inst., Calcutta, 1984.

\bibitem[DR06]{MR2278358}
Persi Diaconis and Silke W.~W. Rolles.
\newblock Bayesian analysis for reversible {M}arkov chains.
\newblock {\em Ann. Statist.}, 34(3):1270--1292, 2006.

\bibitem[DS10]{MR2736958}
M.~Disertori and T.~Spencer.
\newblock Anderson localization for a supersymmetric sigma model.
\newblock {\em Comm. Math. Phys.}, 300(3):659--671, 2010.

\bibitem[DSZ10]{MR2728731}
M.~Disertori, T.~Spencer, and M.~R. Zirnbauer.
\newblock Quasi-diffusion in a 3{D} supersymmetric hyperbolic sigma model.
\newblock {\em Comm. Math. Phys.}, 300(2):435--486, 2010.

\bibitem[DV02]{MR1900324}
Burgess Davis and Stanislav Volkov.
\newblock Continuous time vertex-reinforced jump processes.
\newblock {\em Probab. Theory Related Fields}, 123(2):281--300, 2002.

\bibitem[Dyn80]{MR585179}
E.~B. Dynkin.
\newblock Markov processes and random fields.
\newblock {\em Bull. Amer. Math. Soc. (N.S.)}, 3(3):975--999, 1980.

\bibitem[Dyn83]{MR693227}
E.~B. Dynkin.
\newblock Markov processes as a tool in field theory.
\newblock {\em J. Funct. Anal.}, 50(2):167--187, 1983.

\bibitem[Dyn84]{MR902412}
E.~B. Dynkin.
\newblock Local times and quantum fields.
\newblock In {\em Seminar on stochastic processes, 1983 ({G}ainesville, {F}la.,
  1983)}, volume~7 of {\em Progr. Probab. Statist.}, pages 69--83.
  Birkh\"{a}user Boston, Boston, MA, 1984.

\bibitem[Efe99]{efetov1999supersymmetry}
Konstantin Efetov.
\newblock {\em Supersymmetry in disorder and chaos}.
\newblock Cambridge University Press, 1999.

\bibitem[Eis95]{MR1459468}
Nathalie Eisenbaum.
\newblock Une version sans conditionnement du th\'{e}or{\`e}me d'isomorphisms
  de {D}ynkin.
\newblock In {\em S\'{e}minaire de {P}robabilit\'{e}s, {XXIX}}, volume 1613 of
  {\em Lecture Notes in Math.}, pages 266--289. Springer, Berlin, 1995.

\bibitem[Eis05]{MR2120243}
Nathalie Eisenbaum.
\newblock A connection between {G}aussian processes and {M}arkov processes.
\newblock {\em Electron. J. Probab.}, 10:no. 6, 202--215, 2005.

\bibitem[EKM{\etalchar{+}}00]{MR1813843}
Nathalie Eisenbaum, Haya Kaspi, Michael~B. Marcus, Jay Rosen, and Zhan Shi.
\newblock A {R}ay-{K}night theorem for symmetric {M}arkov processes.
\newblock {\em Ann. Probab.}, 28(4):1781--1796, 2000.

\bibitem[JLM85]{MR782955}
A.~Jaffe, H.~Lehmann, and G.~Mack.
\newblock Kurt {S}ymanzik.
\newblock {\em Comm. Math. Phys.}, 97(1-2):1--4, 1985.

\bibitem[Kni63]{MR154337}
F.~B. Knight.
\newblock Random walks and a sojourn density process of {B}rownian motion.
\newblock {\em Trans. Amer. Math. Soc.}, 109:56--86, 1963.

\bibitem[LJ87]{MR941982}
Y.~Le~Jan.
\newblock Temps local et superchamp.
\newblock In {\em S\'{e}minaire de {P}robabilit\'{e}s, {XXI}}, volume 1247 of
  {\em Lecture Notes in Math.}, pages 176--190. Springer, Berlin, 1987.

\bibitem[LJ10]{MR2675000}
Yves Le~Jan.
\newblock Markov loops and renormalization.
\newblock {\em Ann. Probab.}, 38(3):1280--1319, 2010.

\bibitem[LJ11]{MR2815763}
Yves Le~Jan.
\newblock {\em Markov paths, loops and fields}, volume 2026 of {\em Lecture
  Notes in Mathematics}.
\newblock Springer, Heidelberg, 2011.
\newblock Lectures from the 38th Probability Summer School held in Saint-Flour,
  2008, \'{E}cole d'\'{E}t\'{e} de Probabilit\'{e}s de Saint-Flour.

\bibitem[LJ24]{MR4789605}
Yves Le~Jan.
\newblock {\em Random walks and physical fields}, volume 106 of {\em
  Probability Theory and Stochastic Modelling}.
\newblock Springer, Cham, 2024.

\bibitem[LP16]{MR3616205}
Russell Lyons and Yuval Peres.
\newblock {\em Probability on trees and networks}, volume~42 of {\em Cambridge
  Series in Statistical and Probabilistic Mathematics}.
\newblock Cambridge University Press, New York, 2016.

\bibitem[LTF07]{MR2255196}
Gregory~F. Lawler and Jos\'{e}~A. Trujillo~Ferreras.
\newblock Random walk loop soup.
\newblock {\em Trans. Amer. Math. Soc.}, 359(2):767--787, 2007.

\bibitem[Lup16]{MR3502602}
Titus Lupu.
\newblock From loop clusters and random interlacements to the free field.
\newblock {\em Ann. Probab.}, 44(3):2117--2146, 2016.

\bibitem[Lut83]{MR713539}
J.~M. Luttinger.
\newblock The asymptotic evaluation of a class of path integrals. {II}.
\newblock {\em J. Math. Phys.}, 24(8):2070--2073, 1983.

\bibitem[LW04]{MR2045953}
Gregory~F. Lawler and Wendelin Werner.
\newblock The {B}rownian loop soup.
\newblock {\em Probab. Theory Related Fields}, 128(4):565--588, 2004.

\bibitem[LW17]{letac2017multivariate}
G{\'e}rard Letac and Jacek Weso{\l}owski.
\newblock Multivariate reciprocal inverse {Gaussian} distributions from the
  {Sabot-Tarr\`es-Zeng} integral.
\newblock {\em arXiv preprint arXiv:1709.04843}, 2017.

\bibitem[MRT19]{MR3936158}
Franz Merkl, Silke W.~W. Rolles, and Pierre Tarr{\`e}s.
\newblock Convergence of vertex-reinforced jump processes to an extension of
  the supersymmetric hyperbolic nonlinear sigma model.
\newblock {\em Probab. Theory Related Fields}, 173(3-4):1349--1387, 2019.

\bibitem[PW96]{MR1611693}
James~Gary Propp and David~Bruce Wilson.
\newblock Exact sampling with coupled {M}arkov chains and applications to
  statistical mechanics.
\newblock In {\em Proceedings of the {S}eventh {I}nternational {C}onference on
  {R}andom {S}tructures and {A}lgorithms ({A}tlanta, {GA}, 1995)}, volume~9,
  pages 223--252, 1996.

\bibitem[Ray63]{MR156383}
Daniel Ray.
\newblock Sojourn times of diffusion processes.
\newblock {\em Illinois J. Math.}, 7:615--630, 1963.

\bibitem[ST15]{MR3420510}
Christophe Sabot and Pierre Tarr{\`e}s.
\newblock Edge-reinforced random walk, vertex-reinforced jump process and the
  supersymmetric hyperbolic sigma model.
\newblock {\em J. Eur. Math. Soc. (JEMS)}, 17(9):2353--2378, 2015.

\bibitem[ST16]{MR3520013}
Christophe Sabot and Pierre Tarres.
\newblock Inverting {R}ay-{K}night identity.
\newblock {\em Probab. Theory Related Fields}, 165(3-4):559--580, 2016.

\bibitem[STZ17]{MR3729620}
Christophe Sabot, Pierre Tarr{\`e}s, and Xiaolin Zeng.
\newblock The vertex reinforced jump process and a random {S}chr\"{o}dinger
  operator on finite graphs.
\newblock {\em Ann. Probab.}, 45(6A):3967--3986, 2017.

\bibitem[SW12]{MR2979861}
Scott Sheffield and Wendelin Werner.
\newblock Conformal loop ensembles: the {M}arkovian characterization and the
  loop-soup construction.
\newblock {\em Ann. of Math. (2)}, 176(3):1827--1917, 2012.

\bibitem[Sym68]{symanzik1968euclidean}
Kurt Symanzik.
\newblock Euclidean quantum field theory.
\newblock In {\em Conf. Proc.}, volume 680812, pages 152--226, 1968.

\bibitem[SZ19]{MR3904155}
Christophe Sabot and Xiaolin Zeng.
\newblock A random {S}chr\"{o}dinger operator associated with the vertex
  reinforced jump process on infinite graphs.
\newblock {\em J. Amer. Math. Soc.}, 32(2):311--349, 2019.

\bibitem[Szn12a]{MR2892408}
Alain-Sol Sznitman.
\newblock An isomorphism theorem for random interlacements.
\newblock {\em Electron. Commun. Probab.}, 17:no. 9, 9, 2012.

\bibitem[Szn12b]{MR2932978}
Alain-Sol Sznitman.
\newblock {\em Topics in occupation times and {G}aussian free fields}.
\newblock Zurich Lectures in Advanced Mathematics. European Mathematical
  Society (EMS), Z\"{u}rich, 2012.

\bibitem[Weg16]{wegner2016supermathematics}
Franz Wegner.
\newblock Supermathematics and its applications in statistical physics.
\newblock {\em Lecture Notes in Physics}, 920, 2016.

\bibitem[Wil77]{MR471757}
David Williams.
\newblock The {E}uclidean loop expansion for massive {$\lambda F^{4}_{4}$}:
  through one loop
\newblock {\em Comm. Math. Phys.}, 54:193--218, 1977.

\bibitem[Wil96]{MR1427525}
David~Bruce Wilson.
\newblock Generating random spanning trees more quickly than the cover time.
\newblock In {\em Proceedings of the {T}wenty-eighth {A}nnual {ACM} {S}ymposium
  on the {T}heory of {C}omputing ({P}hiladelphia, {PA}, 1996)}, pages 296--303.
  ACM, New York, 1996.

\bibitem[Zen16]{MR3531700}
Xiaolin Zeng.
\newblock How vertex reinforced jump process arises naturally.
\newblock {\em Ann. Inst. Henri Poincar\'{e} Probab. Stat.}, 52(3):1061--1075,
  2016.

\bibitem[Zir91]{MR1134935}
Martin~R. Zirnbauer.
\newblock Fourier analysis on a hyperbolic supermanifold with constant
  curvature.
\newblock {\em Comm. Math. Phys.}, 141(3):503--522, 1991.

\end{thebibliography}
\newcommand{\etalchar}[1]{$^{#1}$}

\end{document}